\documentclass[a4paper,reqno]{amsart}
\usepackage{amsmath,amsthm,amssymb,bm}
\usepackage{mathscinet}
\usepackage[hypertex]{hyperref}

\usepackage[all]{xy}

\makeatletter
    
    \@addtoreset{equation}{section}
\makeatother

\newtheorem{Thm}{Theorem}[section]
\newtheorem{Lem}[Thm]{Lemma}
\newtheorem{Cor}[Thm]{Corollary}
\newtheorem{Prop}[Thm]{Proposition}
\newtheorem{Rem}[Thm]{Remark}
\newtheorem{Def}[Thm]{Definition}

\newcommand{\C}{\mathbb{C}}           % Use for complex numbers.
\newcommand{\Z}{\mathbb{Z}}
\newcommand{\Q}{\mathbb{Q}}

\newcommand{\ad}{\text{ad}}

\newcommand{\id}{\mathrm{id}}
\newcommand{\ch}{\mathrm{ch}\,}

\newcommand{\fa}{{\mathfrak a}}             % more Fraktur
\newcommand{\fb}{{\mathfrak b}}
\newcommand{\fg}{{\mathfrak g}}
\newcommand{\fh}{{\mathfrak h}}

\newcommand{\fm}{{\mathfrak m}}
\newcommand{\fn}{{\mathfrak n}}
\newcommand{\fp}{{\mathfrak p}}

\newcommand{\hfg}{\hat{\fg}}
\newcommand{\hfh}{\hat{\fh}}
\newcommand{\hfb}{\hat{\fb}}

\newcommand{\hfn}{\hat{\fn}}
\newcommand{\hfp}{\hat{\fp}}

\newcommand{\hP}{\hat{P}}
\newcommand{\hQ}{\hat{Q}}
\newcommand{\hI}{\hat{I}}
\newcommand{\hW}{\hat{W}}

\newcommand{\ga}{\alpha}
\newcommand{\gb}{\beta}

\newcommand{\gl}{\lambda}
\newcommand{\gL}{\Lambda}
\newcommand{\gd}{\delta}
\newcommand{\gD}{\Delta}

\newcommand{\gt}{\theta}

\renewcommand{\ggg}{\gamma}
\newcommand{\gs}{\sigma}
\newcommand{\gS}{\Sigma}

\newcommand{\gee}{\varepsilon}

\renewcommand{\hat}{\widehat}
\newcommand{\ol}{\overline}

\newcommand{\wti}{\widetilde}

\newcommand{\wt}{\mathrm{wt}}

\newcommand{\bL}{\mathbf{L}}
\newcommand{\Lfg}{\bL\fg}
\newcommand{\Lfn}{\bL\fn}
\newcommand{\hV}{\hat{V}}
\newcommand{\cF}{\mathcal{F}}

\newcommand{\chh}{\mathrm{ch}_{\fh}}
\newcommand{\chhh}{\mathrm{ch}_{\hfh}}
\newcommand{\D}{\mathcal{D}}

\newcommand{\chq}{\mathrm{ch}_q}
\newcommand{\hgDre}{\hat{\gD}^{\mathrm{re}}}

\begin{document}
\title[Demazure modules and graded limits of minimal affinizations]{Demazure modules 
      and graded limits of minimal affinizations}
\author{Katsuyuki Naoi}
\date{}

\begin{abstract}
  For a minimal affinization over a quantum loop algebra of type $BC$,
  we provide a character formula in terms of Demazure operators and multiplicities in terms of crystal bases.
  We also prove the formula for the limit of characters conjectured by Mukhin and Young.
  These are achieved by verifying that its graded limit
  (a variant of a classical limit) is isomorphic to some multiple generalization of a Demazure module, 
  and by determining the defining relations of the graded limit.
\end{abstract}

\maketitle

\section{Introduction}\label{Section:Introduction}

Let $\fg$ be a complex simple Lie algebra of rank $n$, 
and $\bL\fg = \fg \otimes \C[t,t^{-1}]$ the associated loop algebra.
The theory of finite-dimensional representations of the quantum loop algebra $U_q(\bL\fg)$ has been intensively studied
from various viewpoints in recent years. For example, see the survey \cite{MR2642561} and references therein.

In \cite{MR1367675}, Chari introduced the notion of minimal affinizations. 
An affinization $\hat{V}$ of a simple $U_q(\fg)$-module $V$ 
is by definition a simple $U_q(\bL\fg)$-module whose highest weight is equal to that of $V$.
%which decomposes as a direct sum of simple $U_q(\fg)$-modules 
%$\hat{V} = V_0 \oplus V_1 \oplus \cdots \oplus V_p$ such that $V_0 \cong V$, and 
%the highest weight of $V_i$ is smaller than that of $V$ if $i > 0$.
Two affinizations of $V$ are said to be equivalent if they are isomorphic as $U_q(\fg)$-modules.
Then one can define a partial ordering on the set of
equivalence classes of affinizations of $V$, and minimal ones with respect to this ordering are called minimal affinizations.
An almost complete classification of minimal affinizations was done by Chari and Pressley 
in \cite{MR1367675,MR1347873,MR1376937,MR1402568},
and in particular it was proved that, if $\fg$ is of type $ABCFG$, for every simple $U_q(\fg)$-module
its minimal affinization is unique.

Given a minimal affinization, one can consider its classical limit.
By restricting it to the current algebra $\fg[t] = \fg \otimes \C[t]$ and taking a pull-back, 
a graded $\fg[t]$-module is obtained. 
In this article we call this the \textit{graded limit}.
Graded limits are quite important for the study of minimal affinizations since
the $U_q(\fg)$-module structure of a minimal affinization is completely determined by the $U(\fg)$-module structure of
its graded limit.
This idea was applied in \cite{MR1836791,MR2238884} to Kirillov-Reshetikhin modules,
which are minimal affinizations whose highest weights are multiples of a fundamental weight.

The graded limits of general minimal affinizations were first studied by Moura in \cite{MR2587436}.
In the article, he defined some two $\fg[t]$-modules using the graded limits of Kirillov-Reshetikhin modules,
and conjectured in all types that the graded limit of a minimal affinization is isomorphic to them.
This conjecture was proved in type $A$ and partially in type $BD$ in the article,
and partially in type $E_6$ in \cite{MR2896463}.

In the present paper we study in more detail the graded limits of minimal affinizations in type $ABC$.
These are the classical types in which minimal affinizations are unique (our main interest is in type $BC$ 
since type $A$ is well-known).
In particular, we give a proof to the Moura's conjecture in these types (though we need some modification in type $C$).

To introduce our results, let us define some $\fg[t]$-modules.
Denote by $\hat{\fg}$ the nontwisted affine Lie algebra associated with $\fg$, 
and by $\hfb \subseteq \hfg$ the standard Borel subalgebra.
Let $\xi_1,\ldots,\xi_p \in \hP$ be a sequence of weights of $\hfg$,
and assume that each $\xi_i$ belongs to the affine Weyl group orbit $\hW \gL^i$ of a dominant integral weight $\gL^i \in \hP^+$.
We define a $\hfb$-module $D(\xi_1, \ldots, \xi_p)$ by 
\[ D(\xi_1, \ldots, \xi_p) = U(\hfb)(v_{\xi_1} \otimes \cdots \otimes v_{\xi_p}) \subseteq \hV(\gL^1) \otimes \cdots \otimes
   \hat{V}(\gL^p),
\]
where $\hat{V}(\gL)$ is the simple highest weight $\hfg$-module with highest weight $\gL$,
and $v_\xi$ is an extremal weight vector with weight $\xi$. 
When $p =1$, $D(\xi_1)$ is called a Demazure module.
It is easily seen that, if each $\xi_i$ is dominant with respect to $\fg$,
then $D(w_\circ\xi_1, \ldots,w_\circ\xi_p)$ is $\fg[t]$-stable, where $w_\circ$ is the longest element of the Weyl group
of $\fg$. 

Assume that $\fg$ is of type $ABC$.
Let $V_q(\gl)$ be the simple $U_q(\fg)$-module with highest weight $\gl \in P^+$,
and $L_q(\bm{\pi})$ a minimal affinization of $V_q(\gl)$
(here $\bm{\pi}$ denotes the $\ell$-highest weight. See Subsection \ref{Subsection:QAA}).
By $L(\bm{\pi})$ we denote its graded limit.
Our first main theorem is the following (Theorem \ref{Thm:Main2}).

\begin{Thm}\label{Thm:intro1}
  As a $\fg[t]$-module, $L(\bm{\pi})$ is isomorphic to $D(w_\circ\xi_1,\ldots,w_\circ\xi_n)$ 
  with suitable $\fg$-dominant $\xi_1,\ldots,\xi_n \in \hP$.  
\end{Thm}

When $\fg$ is of type $AB$ and $\gl = \sum_{i} \gl_i \varpi_i$ ($\varpi_i$ are the fundamental weights of $\fg$),
we set $\xi_i = \gl_i\varpi_i + \gl_i' \gL_0$,
where $\gl_i' = \lceil \gl_i/2 \rceil$ if $\fg$ is of type $B$ and $i = n$, and $\gl_i' = \gl_i$ otherwise.
Here $\gL_0$ is the fundamental weight of $\hfg$ associated with the additional index $0$.
In type $C$, we need to choose $\xi_i$'s in a little more complicated way. 
For the detail see Subsection \ref{Main}.

Let $\fg = \fn_+ \oplus \fh \oplus \fn_-$ be a triangular decomposition of $\fg$,
and denote by $\ga_i$ and $\ga_i^{\vee}$ the simple roots and coroots respectively.
Our second main theorem gives the defining relations of $L(\bm{\pi})$ (Theorem \ref{Thm:Main1}).

\begin{Thm}\label{Thm:Intro2}
  The graded limit $L(\bm{\pi})$ is isomorphic to the cyclic $\fg[t]$-module generated by a nonzero vector
  $v$ with relations
  \begin{align*}
    \fn_+[t]v = 0, \ \ &(h \otimes t^s)v= \gd_{s0}\langle h, \gl \rangle v \ \text{for} \ h \in \fh, s \ge 0,\ \ 
    f_i^{\langle \ga_i^\vee, \gl \rangle+1}v = 0 \ \text{for} \ 1 \le i \le  n, \nonumber \\
    &t^2\fn_-[t]v=0, \ \text{and} \ \ (f_{\ga}\otimes t)v= 0 \ \text{for} \ \ga \in \gD_+^1,
  \end{align*}
  where $\gD_+^1$ is a subset of the positive roots $\gD_+$ defined by
  \[ \gD_+^1 = \Big\{ \ga \in \gD_+ \Bigm| \ga = \sum_{1 \le i \le n} n_i \ga_i \ \text{with} \ n_i \le 1 \ \text{for all} \ 
     i\Big\}.
  \]
\end{Thm}

We easily see that the above theorems are reformulation of the Moura's conjecture in type $AB$.
%More precisely, in type $B$ the modules appearing in the theorems are isomorphic to the ones defined in \cite{MR2587436},
%and therefore the two theorems and his conjecture are equivalent. 
%(In type $C$ they are not, and the module $D(w_\circ\xi_1,\ldots,w_\circ\xi_n)$ is essentially needed
%to formulate Theorem \ref{Thm:intro1} in this type.)
It should also be noted that Theorem \ref{Thm:Intro2}, together with a result of \cite{MR2342293}, gives a proof to
\cite[Conjecture 1.13]{MR2763623} in type $B$.

The module $D(w_\circ \xi_1,\ldots,w_\circ \xi_n)$ in Theorem \ref{Thm:intro1} can be constructed in another way as follows.
For a $\hfb$-submodule $D$ of a $\hfg$-module $V$ and an index $i$, let $F_iD$ be the 
$(\hfb \oplus \C f_i)$-submodule of $V$ generated by $D$,
and for an element $w \in \hW$ with reduced expression $w=s_{i_1} \cdots s_{i_p}$, 
set $F_wD = F_{i_1} \cdots F_{i_p}D$.
One can naturally extend $F_w$ for $w \in \hW \rtimes \gS$ (see Subsection \ref{Demazure}), 
where $\gS$ is the group of Dynkin diagram automorphisms.
%For each $1 \le i \le n$, let $\gL^i$ be the dominant integral weight satisfying $\xi_i \in \hW \gL^i$.
Then we have
\begin{align}\label{eq:intro}
   &D(w_\circ\xi_1,\ldots,w_\circ\xi_n)\\ \nonumber &= F_{w_\circ w_1}\Big(D(\gL^1) \otimes F_{w_2}\Big(D(\gL^2) \otimes 
   \cdots \otimes F_{w_{n-1}}\Big(D(\gL^{n-1}) \otimes F_{w_n}D(\gL^n)\Big)\!\cdots\! \Big)\!\Big)
\end{align}
for suitable $\gL^1,\ldots,\gL^n \in\hP^+$ and $w_1,\ldots,w_n \in \hW \rtimes \gS$ 
(see Subsection \ref{Subsection:Corollaries}).
The character of such a module is given by \cite{MR1887117} 
in terms of Demazure operators $\mathcal{D}_w$ (see Subsection \ref{Demazure}).
Since the character of $L_q(\bm{\pi})$ is equal to that of $L(\bm{\pi})$, we obtain the following character formula
as a corollary of Theorem \ref{Thm:intro1} (Corollary \ref{Cor:Cor2}).

\begin{Cor}
  \[ \ch L_q(\bm{\pi}) = \D_{w_\circ w_{1}}\Big(e^{\gL^1}\cdot \D_{w_2}\Big(e^{\gL^2} \cdots 
      \D_{w_{n-1}}\Big(e^{\gL^{n-1}}\cdot \D_{w_{n}}\big(e^{\gL^n}\big)\Big)\!\cdots\!\Big)\!\Big)\Big|_{e^{\gL_0} 
      = e^\gd= 1},
  \]
  where $\gd$ is the null root.
\end{Cor}

The right-hand side of (\ref{eq:intro}) has a crystal analogue.
Using this, we can express the $U_q(\fg)$-module multiplicities of $L_q(\bm{\pi})$ 
as the numbers of some elements in a crystal basis (Corollary \ref{Cor:crystal}).
We would like to emphasize that the crystal basis appearing here is (essentially) of \textit{finite type} 
(see Remark \ref{Rem}).

On the other hand, we deduce from Theorem \ref{Thm:Intro2} the formula for the limit of 
normalized characters (Corollary \ref{Cor:Cor3}), which was conjectured by Mukhin and Young in \cite[Conjecture 6.3]{MY}.

\begin{Cor}
  Let $J \subseteq \{1,\ldots,n\}$,
  and $\gl^1,\gl^2,\ldots$ be an infinite sequence of elements of $P^+$ such that 
  $\lim_{k \to \infty} \langle \gl^k, \ga_i^\vee \rangle = \infty$ if $i \notin J$ and 
  $\langle \gl^k, \ga_i^\vee\rangle =0$ for all $k$ otherwise. 
  Assume that $L_q(\bm{\pi}^k)$ is a minimal affinization of $V_q(\gl^k)$ for each $k$.
  Then $\lim_{k \to \infty} e^{-\gl^k}\mathrm{ch}\, L_q(\bm{\pi}^k)$ exists, and
  \begin{equation*}
     \lim_{k \to \infty} e^{-\gl^k}\mathrm{ch} \, L_q(\bm{\pi}^k) 
     = \prod_{\ga \in \gD_+\setminus \gD_+^J} \frac{1}{1 - e^{-\ga}}\cdot
       \prod_{\ga \in \gD_+ \setminus \gD_+^{1,J}} \frac{1}{1 -e^{-\ga}},
  \end{equation*}  
  where $\gD_+^{J} = \gD_+ \cap \Big(\sum_{i \in J} \Z \ga_i\Big)$ 
  and $\gD_+^{1,J} = \{ \ga \in \gD_+ \mid \ga = \sum_{i} n_i\ga_i \ \text{with} \ n_i \le 1 \ 
  \text{if} \ i \notin J\}$.
\end{Cor}

It should be noted that the existence of the limit has already been proved in \cite{MR2982441}, \cite{MY}.
%It should be noted that the existence of the limits in \cite{MR2982441} Hernandez and Jimbo studied the limits of 
%Kirillov-Reshetikhin modules in the module level, and gave a formula for their $q$-characters. 

The theorems are established by showing one by one the existence of three surjective homomorphisms
\[ D(w_\circ \xi_1,\ldots, w_\circ \xi_n) \twoheadrightarrow M(\gl), \ \ \ M(\gl) \twoheadrightarrow L(\bm{\pi}), \ \ \
   L(\bm{\pi}) \twoheadrightarrow D(w_\circ \xi_1,\ldots,w_\circ \xi_n),
\]
where $M(\gl)$ denotes the $\fg[t]$-module defined in Theorem \ref{Thm:Intro2}.
The key idea to verify the first one is to determine the defining relations of  $D(w_\circ\xi_1,\ldots,w_\circ \xi_n)$ 
inductively using the equality (\ref{eq:intro}).
A main tool to prove the latter two is the theory of $q$-characters introduced by Frenkel and Reshetikhin \cite{MR1745260}.
A $q$-character is a generalization of a usual character which 
records the dimensions of generalized eigenspaces (i.e.,\ $\ell$-weight spaces) with respect to the 
commutative subalgebra $U_q(\bL \fh)$.

In this article we concentrate only on the type $ABC$.
However, (at least a part of) these results would hold in the other types.
These will be studied in future publications.

It should be mentioned that a module similar to the right-hand side of (\ref{eq:intro}) 
also appears in another study of graded limits.
In \cite{MR2964614}, it was proved that the fusion product of the graded limits of Kirillov-Reshetikhin modules is isomorphic
to such a module. This fact was essentially used to prove the $X=M$ conjecture in type $AD$.

The plan of this article is as follows.
In Section 2, after fixing some notation we define modules $D(\xi_1,\ldots,\xi_p)$ and
study their properties.
In Section 3, we review the theory of finite-dimensional representations of a quantum loop algebra.
In Section 4, we state our main theorems and corollaries.
Finally in Section 5, we establish our main theorems by showing the existence of three surjective homomorphisms.
For this we need some results on $q$-characters, which are also recalled in this section.\\ \\ \\

\noindent \textbf{Index of notation}  \\[-6pt]

We provide for the reader's convenience a brief index of the notation which is used repeatedly in this paper:\\

\noindent 2.1: $C=(c_{ij})_{1\le i,j \le n}$, $I$, $d_1,\ldots,d_n$, $\fg$, $\fh$, $\fb$, $\gD$, $\gD_+$,
$\ga_i$, $\varpi_i$, $P$, $P^+$, $Q$, $Q^+$, $W$, \\
\hspace{16pt} $w_\circ$, $e_\ga$, $f_\ga$, $\ga^{\vee}$ ($\ga \in \gD$), $\fn_\pm$, $\fg_J$, $\hfg$, $K$, $d$,
$\hfh$, $\hfb$, $\hat{\gD}$, $\hat{\gD}_+$, $\hgDre$, $\hgDre_+$, $\gd$, $\hI$, $\ga_0$, $e_0$,\\
\hspace{16pt}  $f_0$, $\ga^\vee$ ($\ga \in \hgDre$), $\gL_0$, $\hP$, $\hP^+$, $\hQ$, $\hQ^+$, $\gl \le \nu$, $\hW$, $\gS$, 
$\wti{W}$, $\bL\fa$, $\fa[t]$, $t^s\fg[t]$, $V(\gl)$,\\ 
\hspace{16pt} $V(\gl,a)$, $\chh V$.\\
\noindent 2.2: $\hV(\gL)$, $v_\xi$, $D(\xi_1,\ldots,\xi_p)$, $\hfp_i$, $F_i$, $F_w$, $\mathcal{D}_w$, $F_\tau$.\\
\noindent 3.1: $q_i$, $U_q(\bL\fg)$, $x_{i,r}^\pm$, $k_i^\pm$, $h_{i,m}$, $U_q(\bL\fn_\pm)$, $U_q(\bL\fh)$, $U_q(\fg)$, 
$U_q(\bL\fg_J)$, $U_q(\bL\fh_J)$.\\
\noindent 3.2: $V_q(\gl)$, $P_q^+$, $\bm{\varpi}_{i,a}$, $P_q$, $\wt$, $V_{\bm{\rho}}$, $L_q(\bm{\pi})$, $v_{\bm{\pi}}$,
$\bm{\pi}^*$, ${}^*\!\bm{\pi}$.\\
\noindent 3.3: $\bm{\pi}_{m,a}^{(i)}$, $P_{q,J}$, $\bm{\rho}_J$.\\
\noindent 3.4: $\bf{A}$, $\big(x_{i,r}^\pm\big)^{(k)}$, $U_{\bf{A}}(\bL\fg)$, $P_{\bf{A}}^+$, $L_{\bf{A}}(\bm{\pi})$, 
$\ol{L_q(\bm{\pi})}$.\\
\noindent 4.1: $L(\bm{\pi})$, $\ol{v}_{\bm{\pi}}$.\\
\noindent 4.2: $\xi_1,\ldots,\xi_n$, $i^\flat$, $p_i$, $\gD_+^1$.\\
\noindent 4.3: $w_1,\ldots,w_n$, $w_{[r,t]}$, $\gL^1,\ldots,\gL^n$.\\
\noindent 5.1: $M(\gl)$, $\ga_{p,q}$, $v_M$, $D$, $v_D$.\\
\noindent 5.2: $\wt_\ell (V)$, $\chq V$, $\bm{\ga}_{i,a}$, $\bm{\nu} \le \bm{\rho}$.

\section{Lie algebras}\label{Section:Lie}

\subsection{Notation and basics}

Let $C = (c_{ij})_{1 \le i,j \le n}$ be a Cartan matrix of finite type, and set $I = \{1, \ldots,n\}$. 
Throughout this paper, we assume that the indices are ordered as in \cite[Section 4.8]{MR1104219}.
Let $D = \mathrm{diag}(d_1,\ldots,d_n)$ be the diagonal matrix 
such that $DC$ is symmetric and the numbers $d_1,\ldots,d_n$ are coprime positive integers.

Let $\fg$ be the complex simple Lie algebra associated with $C$.
Fix a Cartan subalgebra $\fh$ and a Borel subalgebra $\fb$ containing $\fh$.
Let $\gD$ be the root system, and $\gD_+$ the set of positive roots.
Denote by $\ga_i$ ($i \in I$) the simple roots and by $\varpi_i$ ($i \in I$) the fundamental weights.
For notational convenience, we set $\varpi_0 = 0$.
Let $P$ be the weight lattice, $P^+$ the set of dominant integral weights,
$Q$ the root lattice and $Q^+ = \sum_{i \in I} \Z_{\ge 0} \ga_i$.
%For $\ga = \sum_{i \in I} n_i \ga_i\in Q^+$, we call the number $\sum_i n_i$ its \textit{height}. 
Let $W$ be the Weyl group and $w_\circ\in W$ the longest element.

For each $\ga \in \gD$, denote by $\fg_\ga$ the corresponding root space,
and fix nonzero elements $e_\ga \in \fg_\ga$, $f_\ga \in \fg_{-\ga}$ and $\ga^\vee \in \fh$ such that
\[ [e_\ga, f_\ga] = \ga^\vee, \ \ \ [\ga^\vee, e_\ga] = 2 e_\ga, \ \ \ [\ga^\vee, f_\ga] = -2 f_\ga.
\]
We also use the notation $e_i = e_{\ga_i}$, $f_i = f_{\ga_i}$ for $i \in I$.
Set $\fn_{\pm} = \bigoplus_{\ga \in \gD_+} \fg_{\pm \ga}$.
For a subset $J \subseteq I$, denote by $\fg_J$ the semisimple Lie subalgebra of $\fg$ 
generated by $\{e_i, f_i \mid i \in J\}$.

Let $\gt\in \gD_+$ be the highest root, and
denote by $( \ , \ )$ the unique non-degenerate invariant symmetric bilinear form on $\fg$ normalized so that 
$(\gt^\vee,\gt^\vee) = 2$.
The restriction of this bilinear form on $\fh$ induces a linear isomorphism $\nu\colon \fh \to \fh^*$.
By $( \ , \ )$ we also denote the bilinear form on $\fh^*$ induced by $\nu^{-1}$.
Then for $\ga = \sum_{i \in I} n_i \ga_i^\vee \in \gD$, it follows that
\[ \ga^\vee = \sum_{i} \frac{(\ga_i,\ga_i)}{(\ga,\ga)}n_i \ga_i^\vee.
\]

Let $\hat{\fg}$ be the non-twisted affine Lie algebra associated with $\fg$:
\[ \hat{\fg} = \fg \otimes \C[t, t^{-1}] \oplus \C K \oplus \C d,
\]
where $K$ denotes the canonical central element and $d$ is the degree operator.
The Lie bracket of $\hat{\fg}$ is given by 
\begin{align*} [x \otimes t^m + & a_1 K + b_1 d, y \otimes t^n + a_2 K + b_2 d] \\
               &= [x, y] \otimes t^{m+n} + n b_1 y \otimes t^{n} - m b_2 x \otimes t^{m} + m\gd_{m, -n} (x,y)K.
\end{align*}
Naturally $\fg$ is regarded as a Lie subalgebra of $\hat{\fg}$.
A Cartan subalgebra $\hat{\fh}$ and a Borel subalgebra $\hat{\fb}$ are defined 
as follows:
\[ \hat{\fh} = \fh \oplus \C K \oplus \C d, \ \ \ \hat{\fb} = \hat{\fh} \oplus \fn_+ \oplus \fg \otimes t\C[t].
\] 
Set $\hat{\fn}_+ = \fn_+ \oplus \fg \otimes t\C[t]$.

We often consider $\fh^*$ as a subspace of $\hat{\fh}^*$ by setting $\langle K, \gl \rangle =  \langle d, \gl \rangle = 0$
for $\gl \in \fh^*$.
Let $\hat{\gD}$ be the root system of $\hat{\fg}$, $\hat{\gD}_+$ the set of positive roots,
$\hat{\gD}^{\mathrm{re}}$ the set of real roots, and 
$\hat{\gD}^{\mathrm{re}}_+ = \hat{\gD}^{\mathrm{re}} \cap \hat{\gD}_+$.
Denote by $\gd$ the indivisible imaginary root in $\hat{\gD}_+$.
Set $\hat{I} = I \sqcup \{ 0 \}$, $\ga_0 = \gd - \gt$,
%Then $\{ \ga_i \mid i \in \hI \}$ is the set of simple roots.
$e_0 = f_\gt \otimes t$ and $f_0 = e_\gt \otimes t^{-1}$.
For $\ga = \gb + s \gd \in \hat{\gD}^{\mathrm{re}}$ with $\gb \in \gD$ and $s \in \Z$, define $\ga^\vee \in \hfh$ by
\[ \ga^{\vee} = \gb^{\vee} + \frac{2s}{(\gb,\gb)}K.
\]
Denote by $\gL_0 \in \hfh^*$ the unique element satisfying $\langle K, \gL_0 \rangle = 1$ and 
$\langle \fh, \gL_0 \rangle = \langle d, \gL_0 \rangle = 0$, and define $\hP, \hP^+ \subseteq \hfh^*$ by 
\[ \hP = P \oplus \Z \gL_0 \oplus \C \gd \ \ \ \text{and} \ \ \ \hP^+ = \big\{ \xi \in \hP \mid \langle \ga_i^\vee, \xi \rangle
   \ge 0 \ \text{for all} \ i \in \hI \,\big\}.
\]
Let $\hQ = \sum_{i \in \hI} \Z \ga_i$ and $\hQ^+ = \sum_{i \in \hI} \Z_{\ge 0} \ga_i$.
For $\xi_1, \xi_2 \in \hP$, we write $\xi_1 \le \xi_2$ if $\xi_2 - \xi_1 \in \hQ^+$.
%For $\gL \in \hP$, the \textit{level} of $\gL$ is defined by the integer $\langle K, \gL \rangle$.
Let $\hat{W}$ be the Weyl group of $\hat{\fg}$, and regard $W$ naturally as a subgroup of $\hat{W}$. 
Let $\ell\colon \hW \to \Z_{\ge 0}$ be the length function.
Denote by $\gS$ the group of Dynkin diagram automorphisms of $\hfg$.
A linear action of $\gS$ on $\hfh^*$ is defined by letting $\tau \in \gS$ act as follows:
\[ \tau(\ga_i) = \ga_{\tau(i)} \ \ \text{for} \ i \in \hI, \ \ \ \tau(\gL_0) = \varpi_{\tau(0)} + \gL_0 - 
   \frac{1}{2}\big( \varpi_{\tau(0)}, \varpi_{\tau(0)}\big) \gd.
\]
Let $\wti{W}$ be the subgroup of $GL(\hfh^*)$ generated by $\hW$ and $\gS$. 
Since $\tau s_i =s_{\tau(i)} \tau$ holds for $\tau \in \gS$ and $i \in \hI$, we have $\wti{W} = \hW \rtimes \gS$.
We also define an action of $\gS$ on $\hfg$ by letting $\tau \in \gS$ act as the Lie algebra automorphism defined by
\[ \tau(e_i) = e_{\tau(i)}, \ \ \tau(f_i) = f_{\tau(i)} \ \ \text{for} \ \ i \in \hI
   \ \ \text{and} \ \ \big\langle \tau(h), 
   \tau(\gl)\big\rangle = \big\langle h, \gl \big\rangle \ \ \text{for} \ h \in \hfh, \ \gl \in \hfh^*.
\]
%where $\varpi_j^{\vee}$ is the fundamental coweight of $\fg$ if $j \in I$, and $\varpi_0^{\vee} = 0$.
The length function $\ell$ is extended on $\wti{W}$ by setting $\ell(w\tau) = \ell(w)$ for $w \in \hW,\tau \in \gS$.

Given a Lie algebra $\mathfrak{a}$, its \textit{loop algebra} $\bL\fa$ is defined by the tensor product
$\fa \otimes \C[t,t^{-1}]$ equipped with the Lie algebra structure given by $[x \otimes f, y \otimes g] = [x,y]\otimes fg$.
Let $\fa[t]$ and $t^s \fa[t]$ for $s \in \Z_{> 0}$ denote the Lie subalgebras $\fa \otimes \C[t]$ and $\fa \otimes t^s \C[t]$
respectively.
The Lie algebra $\fa[t]$ is called the \textit{current algebra} associated with $\fa$.
%Note that $\hfg$ contains $\fg[t]$ as a Lie subalgebra.

Denote by $V(\gl)$ the simple $\fg$-module with highest weight $\gl \in P^+$.
For $a \in \C^\times$, let $\text{ev}_a\colon \bL\fg \to \fg$ denote the \textit{evaluation map} defined by
$\text{ev}_a(x \otimes f) = f(a)x$.
Denote by $V(\gl,a)$ the simple $\bL\fg$-module defined by the pull-back of $V(\gl)$ with respect to $\text{ev}_a$,
which is called an \textit{evaluation module}. 
An evaluation module for $\fg[t]$ is similarly defined, and also denoted by $V(\gl,a)$ ($\gl \in P^+, a \in \C$).

For a finite-dimensional semisimple $\fh$-module $V$, define the $\fh$-character $\chh V$ by  
\[ \chh V = \sum_{\gl \in \fh^*} e^{\gl}\dim V_{\gl}  \in \Z[\fh^*],
\]
where $V_\gl=\{ v \in V \mid hv = \langle h, \gl \rangle v \ \text{for} \ h \in \fh\}$. 
For a finite-dimensional semisimple $\hfh$-module $\hV$, the $\hfh$-character $\chhh \hV \in \Z[\hfh^*]$ is defined similarly.
We will omit the subscript $\fh$ or $\hfh$ when they are obvious from the context.

\subsection{Demazure modules and generalizations}\label{Demazure}

For each $\xi \in \hW(\hP^+)$, we define a $\hfb$-module $D(\xi)$ as follows:
let $\gL$ be the unique element of $\hP^+$ such that $\xi \in \hW\gL$,
and denote by $\hV(\gL)$ the simple highest weight $\hfg$-module with highest weight $\gL$.
Let $v_\xi \in \hV(\gL)$ be an extremal weight vector with weight $\xi$,
and set $D(\xi) = U(\hfb)v_{\xi} \subseteq \hV(\gL)$.

\begin{Def} \normalfont
  The $\hfb$-module $D(\xi)$ is called a \textit{Demazure module}. 
\end{Def}

Note that, for $i \in \hI$, $D(\xi)$ is $f_i$-stable if and only if $\langle \ga_i^\vee, \xi \rangle \le 0$.
In this article we consider the following generalization of a Demazure module.
Let $\xi_1, \ldots,\xi_p$ be a sequence of elements of $\hW(\hP^+)$.
For each $1 \le j \le p$, let $\gL^j$ be the element of $\hP^+$ satisfying $\xi_j \in \hW \gL^j$,
and define a $\hfb$-submodule $D(\xi_1,\ldots,\xi_p)$ of $\hV(\gL^1) \otimes \cdots \otimes \hV(\gL^p)$ by
\[ D(\xi_1,\ldots,\xi_p) = U(\hfb)(v_{\xi_1} \otimes \cdots \otimes v_{\xi_p}).
\]
If $\langle \ga_i^\vee, \xi_j \rangle \le 0$ holds for all $1\le j \le p$, then
$D(\xi_1,\ldots,\xi_p)$ is $f_i$-stable.

Though it seems difficult to give characters of $D(\xi_1,\ldots,\xi_p)$ in general,
when the sequence $\xi_1, \ldots, \xi_p$ has some special property,
the character is given in terms of Demazure operators.
To explain this, let us recall a result in \cite{MR1887117}.
Denote by $\hfp_i$ for $i \in \hI$ the parabolic subalgebra $\hfb \oplus \C f_i$.
For a $\hfg$-module $V$, a $\hfb$-submodule $D$ of $V$ and $i \in \hI$,
let $F_{i}D = U(\hfp_i)D \subseteq V$.
For $w \in \hW$ with reduced expression $w = s_{i_1} \cdots s_{i_k}$,
we set 
\[ F_wD = F_{i_1}\cdots F_{i_k}D.
\]
Though the definition of $F_w$ depends on the choice of a reduced expression,
we will use this by abuse of notation
(most of the modules $F_w D$ in this article do not depend on the choices).
For $i \in \hI$, define a linear operator $\mathcal{D}_{i}$ on $\Z[\hP]$ by
\begin{equation*}
   \mathcal{D}_{i}(f) = \frac{f - e^{-\ga_i}\cdot s_i(f)}{1-e^{-\ga_i}},
\end{equation*}
where $s_i$ acts on $\Z[\hP]$ by $s_i(e^{\xi}) = e^{s_i\xi}$.
The operator $\mathcal{D}_{i}$ is called the \textit{Demazure operator} associated with $i$.
For $w \in \hW$ and its reduced expression $w = s_{i_1}\cdots s_{i_k}$, 
the operator $\D_w = \D_{i_1} \cdots \D_{i_k}$ is independent
of the choice of the expression \cite{MR1923198}.
The following theorem is a reformulation of \cite[Theorem 5]{MR1887117} for our setting 
(note that $D(\gL)$ is $1$-dimensional if $\gL \in \hP^+$):

\begin{Thm}\label{Thm:LLM}
  For sequences $\gL^1, \ldots,\gL^p$ of elements of $\hP^+$ and $w_1,\ldots,w_p$ of elements of $\hW$,
  we have
  \begin{align}\label{eq:LLM}
    \chhh F_{w_1}\Big(D&(\gL^1) \otimes F_{w_2}\Big(D(\gL^2) \otimes \cdots \otimes
    F_{w_{p-1}}\Big(D(\gL^{p-1}) \otimes F_{w_p}D(\gL^p) \Big) \!\cdots\! \Big)\!\Big) \nonumber \\
    &= \mathcal{D}_{w_{1}}\Big( e^{\gL_1}\cdot \D_{w_{2}}\Big(e^{\gL^2} \cdots \D_{w_{p-1}}
      \Big(e^{\gL_{p-1}} \cdot \mathcal{D}_{w_{p}}(e^{\gL_p})\Big)\!\cdots\!\Big)\!\Big).
  \end{align}
\end{Thm}

\begin{Rem}\normalfont
  In \cite{MR1887117}, the authors studied $\hfb$-modules $V_{\bf i,m}$ called \textit{generalized Demazure modules}.
  The $\hfb$-module in the left-hand side of (\ref{eq:LLM}) is easily identified with 
  a generalized Demazure module (see [loc.\ cit., Subsection 1.1], in which the authors explain how
  a Demazure module is identified with a generalized Demazure module).
  Under this identification, the above equality follows from [loc.\ cit., Theorem 5].
\end{Rem}

In some cases, we can construct $D(\xi_1,\ldots,\xi_p)$ using $F_w$'s.
To see this, we need the following lemma.

\begin{Lem} \label{Lem:one_change}
  Let $\xi_1,\ldots,\xi_p$ be a sequence of elements of $\hW(\hP^+)$ and $i \in \hI$. 
  If $\langle \ga_i^\vee, \xi_j \rangle \ge 0$ holds
  for all $1 \le j \le p$, then we have
    \[ F_i D(\xi_1,\ldots,\xi_p) = D(s_i \xi_1,\ldots,s_i\xi_p).
    \]
\end{Lem}

\begin{proof}
  Let $\mathfrak{sl}_{2,i}$ be the Lie subalgebra of $\hfg$ spanned by $\{e_i, \ga_i^\vee, f_i\}$.
  Since $e_iv_{\xi_j} =0$ and $f_i v_{s_i\xi_j} = 0$ hold for all $j$, we easily see that
  \[ U(\mathfrak{sl}_{2,i})(v_{\xi_1} \otimes \cdots \otimes v_{\xi_p}) 
     = U(\mathfrak{sl}_{2,i})(v_{s_i\xi_1} \otimes \cdots \otimes v_{s_i\xi_p}).
  \]
  Since $D(s_i \xi_1,\ldots,s_i\xi_p)$ is $\hfp_i$-stable, this implies the assertion. 
\end{proof}

Let $w_{1}, \ldots, w_{p}$ be a sequence of elements of $\hW$, 
and denote by $w_{[r,t]}$ the element $w_{r}w_{r+1} \cdots w_{t} \in \hW$ for $1 \le r \le t \le p$.
We assume that $\ell(w_{[1,p]}) = \sum_{j=1}^p \ell(w_{j})$.
Then for every $1 \le r \le t \le p$ and $\gL \in \hP^+$, if $w_{r} = s_{i_1} \cdots s_{i_{N(r)}}$ is a reduced expression of 
$w_r$, then 
\[ \langle \ga_{i_{k}}^\vee, s_{i_{k+1}} \cdots s_{i_{N(r)}}w_{[r+1,t]} \gL \rangle \ge 0
\]
holds for all $1 \le k \le N(r)$ since $\ell(w_{[r,t]}) = \sum_{j =r}^t \ell(w_{j})$.
Hence by applying Lemma \ref{Lem:one_change} several times, the following proposition is proved.

\begin{Prop}\label{Prop:isom}
  Let $\gL^1, \ldots, \gL^p$ be a sequence of elements of $\hP^+$, and $w_{1}, \ldots, w_{p}$ a sequence of elements of
  $\hW$ such that $\ell(w_{[1,p]}) = \sum_{j=1}^p \ell(w_{j})$.
  Then we have
  \begin{align*}\label{eq:multi_change}
      F_{w_{1}}\Big(D(\gL^1) \otimes F_{w_{2}}\Big(&D(\gL^2)\otimes \cdots
      \otimes F_{w_{p-1}}\Big(D(\gL^{p-1}) \otimes F_{w_{p}} 
       D(\gL^p)\Big)\!\cdots\! \Big)\!\Big) \nonumber \\
     &= D\Big(w_{[1,1]} \gL^1, w_{[1,2]} \gL^2, \ldots, w_{[1,p-1]}\gL^{p-1}, w_{[1, p]}\gL^p\Big). 
  \end{align*} 
\end{Prop}

In conclusion, if there is a sequence $w_1,\ldots,w_p \in \hW$ satisfying $\xi_j = w_{[1,j]}\gL^j$ with
$\gL^j \in \hP^+$ and $\ell (w_{[1,p]})=\sum_{j = 1}^p \ell(w_j)$,
then the character of $D(\xi_1,\ldots,\xi_p)$ is obtained from Theorem \ref{Thm:LLM} and Proposition 2.5.

For later use, we need to generalize the above results for elements of $\wti{W}$.
Let $\tau$ be an element of $\gS$. 
For $\gL \in \hP^+$, we define a linear map $F_\tau\colon \hV(\gL) \to \hV(\tau\gL)$ by
$F_\tau(xv_\gL) = \tau(x)v_{\tau\gL}$ for $x \in U(\hfg)$, which is well-defined and 
bijective by \cite[Corollary 10.4]{MR1104219}.
Note that if $v \in \hV(\gL)_\xi$ with $\xi \in \hP$, then $F_\tau(v) \in \hV(\tau\gL)_{\tau \xi}$.
For an arbitrary sequence $\gL^1,\ldots,\gL^p$ of elements of $\hP^+$,
the tensor product $F_\tau^{\otimes p}$ defines an linear isomorphism $
\hV(\gL^1) \otimes \cdots \otimes \hV(\gL^p) \to \hV(\tau\gL^1) \otimes \cdots \otimes \hV(\tau\gL^p)$,
which we also denote by $F_\tau$ instead of $F_\tau^{\otimes p}$.
If $D$ is a $\hfb$-submodule of $\hV(\gL^1) \otimes \cdots \otimes \hV(\gL^p)$,
the image $F_\tau D$ is also a $\hfb$-submodule.
%For $\tau \in \gS$ and a $\hfg$-module $V$, let us denote by $F_\tau V$ the $\hfg$-module $\{v^{\tau} \mid v \in V\}$
%with $\tau(x)v^{\tau} = (xv)^{\tau}$ for $x \in \hfg$ and $v \in V$.
%For a $\hfb$-submodule $D$ of $V$, let $F_\tau D$ denote the $\hfb$-submodule $\{v^{\tau} \mid v \in D \}$ of $F_{\tau} V$.
The following lemma is easily proved.

\begin{Lem}\label{Lem:tau}
  For a sequence $\xi_1, \ldots, \xi_p$ of elements of $\hW(\hP^+)$, we have
  \[ F_\tau  D(\xi_1,\ldots,\xi_p) = D(\tau\xi_1, \ldots,\tau\xi_p).
  \]
\end{Lem} 

%\begin{proof}
%  It is easy to see that $F_{\tau}\hV(\gL) \cong \hV(\tau\gL)$ for $\gL \in \hP^+$.
%  Let $\gL^1, \ldots, \gL^p \in \hP^+$ be the elements such that $\xi_j \in \hW \gL^j$.
%  We have
%  \begin{align*}
%     F_{\tau}\big( \hV(\gL^1) \otimes \cdots \otimes \hV(\gL^p)\big)&\cong F_{\tau}\hV(\gL^1) \otimes \cdots 
%     \otimes F_{\tau}\hV(\gL^p) \\ &\cong \hV(\tau\gL^1) \otimes \cdots \otimes \hV(\tau\gL^p),
%  \end{align*}
%  and this isomorphism maps $(v_{\xi_1} \otimes \cdots \otimes v_{\xi_p})^\tau$ to
%  a nonzero scalar multiple of $v_{\tau\xi_1} \otimes \cdots \otimes v_{\tau \xi_p}$.
%  Hence the assertion follows.
%\end{proof}

For a $\hfb$-submodule $D$, it is easily seen that $F_\tau F_{i} D= F_{\tau(i)}F_\tau D$. 
Set $F_{w\tau}D = F_w F_\tau D$ for $w \in \hW$ and $\tau \in \gS$.
Now the following proposition is an easy generalization of Proposition \ref{Prop:isom}
($w_{[r,t]}$ are defined as above).

\begin{Prop}\label{Prop:character2}
  Let $\gL^1, \ldots, \gL^p$ be a sequence of elements of $\hP^+$, and $w_{1}, \ldots, w_{p}$ a sequence of elements of
  $\wti{W}$ such that $\ell(w_{[1,p]}) = \sum_{j=1}^p \ell(w_{j})$.
  Then we have
  \begin{align*}
     F_{w_{1}}\Big(D(\gL^1) \otimes F_{w_{2}}\Big(D&(\gL^2)\otimes \cdots
      \otimes F_{w_{p-1}}\Big(D(\gL^{p-1}) \otimes F_{w_{p}} 
       D(\gL^p)\Big)\!\cdots\! \Big)\!\Big) \nonumber \\
     &= D\Big(w_{[1,1]} \gL^1, w_{[1,2]} \gL^2, \ldots, w_{[1,p-1]}\gL^{p-1}, w_{[1, p]}\gL^p\Big).
  \end{align*} 
\end{Prop}

For $\tau \in \gS$, define a linear operator $\mathcal{D}_\tau$ on $\Z[\hP]$ by $\mathcal{D}_{\tau}(e^\xi) =e^{\tau\xi}$.
Obviously $\chhh F_\tau D = \mathcal{D}_\tau \chhh D$ holds,
and we have $\D_{\tau} \D_{i} = \D_{\tau(i)} \D_{\tau}$ \cite[Lemma 4]{MR2235341}.
Set $\D_{w\tau} = \D_w\D_{\tau}$ for $w \in \hW$ and $\tau \in \gS$.
Now the following corollary is obvious from Theorem \ref{Thm:LLM}.

\begin{Cor}\label{Prop:character3}
  The equality {\normalfont(\ref{eq:LLM})} also holds for a sequence $w_1,\ldots,w_p$ of elements of $\wti{W}$ 
  {\normalfont(}instead of $\hat{W}${\normalfont)}.
\end{Cor}

\section{Quantum loop algebras}

\subsection{Definitions and basics}

Let $\C(q)$ denote the ring of rational functions in an indeterminate $q$.
Set $q_i = q^{d_i}$ for $i \in I$,
and 
\[ [l]_{q_i} = \frac{q_i^l - q_i^{-l}}{q_i-q_i^{-1}}, \ \ \ [s]_{q_i}! = [s]_{q_i} [s-1]_{q_i} \cdots [1]_{q_i}, \ \ \
   \begin{bmatrix} s \\ s' \end{bmatrix}_{\!q_i} = \frac{[s]_{q_i}!}{[s-s']_{q_i}![s']_{q_i}!}
\]
for $l \in \Z$ and $s, s' \in \Z_{\ge 0}$ with $s \ge s'$.
The quantum loop algebra $U_q(\bL\fg)$ is the associative $\C(q)$-algebra with generators
\[ x_{i,r}^{\pm} \ (i \in I, r \in \Z), \ \ \ k_i^{\pm 1} \ ( i \in I), \ \ \ h_{i,m} \ (i \in I, m \in \Z \setminus \{0\})
\]  
and the following relations ($i,j \in I, r,r' \in \Z, m,m' \in \Z\setminus \{ 0\})$:
\begin{align*}
  &k_ik_i^{-1} = k_i^{-1}k_i =1, & &[k_i,k_j] = [k_i, h_{j,m}] = [h_{i,m},h_{j,m'}] = 0, \\
  &k_i x_{j,m}^\pm k_i^{-1} = q_i^{\pm c_{ij}}x_{j,m}^\pm,  & 
  &[h_{i,m}, x_{j,r}^{\pm}]= \pm \frac{1}{m}[m c_{ij}]_{q_i}x_{j,r+m}^{\pm}, \\[-23pt]
\end{align*}
\begin{align*}
  [x_{i,r}^{+}, x_{j,r'}^{-}] = \gd_{ij} &\frac{\phi_{i,r+r'}^+ - \phi_{i,r+r'}^-}{q_i - q_i^{-1}},\\
  x_{i,r+1}^{\pm} x_{j,r'}^{\pm} -q_i^{\pm c_{ij}}x_{j,r'}^\pm x_{i,r+1}^\pm 
  &= q_i^{\pm c_{ij}}x_{i,r}^{\pm} x_{j,r'+1}^{\pm} - x_{j,r'+1}^{\pm}x_{i,r}^\pm,\\
  \sum_{\gs \in \mathfrak{S}_s} \sum_{k=0}^{s} (-1)^k \begin{bmatrix} s \\ k \end{bmatrix}_{\!q_i} 
    x_{i, r_{\gs(1)}}^\pm \cdots x_{i, r_{\gs(k)}}^\pm &x_{j,r'}^\pm x_{i,r_{\gs(k+1)}}^\pm \cdots x_{i,r_{\gs(s)}}^\pm
    = 0 \ \ \ \ \ \text{if} \ i \neq j
\end{align*}
for all sequences of integers $r_1, \ldots,r_s$, where $s = 1 - c_{ij}$, $\mathfrak{S}_s$ is the symmetric group on $s$ letters,
and $\phi_{i,r}^\pm$'s are determined by equating coefficients of powers of $u$ in the formula
\[ \sum_{r = 0}^{\infty} \phi_{i, \pm r}^\pm u^{\pm r} = k_i^{\pm 1} \exp \left( \pm(q_i - q_i^{-1})\sum_{r' = 1}^{\infty}
   h_{i, \pm r'} u^{\pm r'} \right),
\]
and $\phi_{i,\mp r}^{\pm} = 0$ for $r > 0$.
The algebra $U_q(\bL\fg)$ is isomorphic to the quotient of the quantum affine algebra 
$U_q'(\hfg)$ by the ideal generated by a certain central element 
\cite{MR914215,MR1301623}.
Denote by $U_q(\bL\fn_{\pm})$ and $U_q(\bL\fh)$ the subalgebras of $U_q(\bL\fg)$ generated by 
$\{x_{i,r}^\pm \mid i \in I, r \in \Z \}$ and 
$\big\{k^{\pm 1}_i,h_{i,m} \mid i \in I, m \in \Z \setminus \{ 0 \} \big\}$ respectively.
Then we have 
\begin{equation}\label{eq:triangular}
  U_q(\Lfg) = U_q(\bL\fn_-) U_q(\bL\fh) U_q(\bL\fn_+) 
\end{equation}
by the existence of a Poincar\'{e}-Birkhoff-Witt type basis \cite{MR1301623}.
Denote by $U_q(\fg)$ the subalgebra generated by $\{x_{i,0}^{\pm}, k^{\pm 1}_i \mid i \in I\}$, which is isomorphic to
the quantized enveloping algebra associated with $\fg$.
For a subset $J \subseteq I$, let $U_q(\bL\fg_J)$ denote the subalgebra generated by 
$\{k_i^{\pm 1}, h_{i,r},x_{i,s}^{\pm}\mid i\in J, r \in \Z \setminus \{0\}, s \in \Z\}$.
We also define $U_q(\bL\fh_J)$ in an obvious way.
When $J = \{i\}$ for some $i \in I$, we simply write $U_q(\bL\fg_i)$ and $U_q(\bL\fh_i)$.

The algebra $U_q(\bL\fg)$ has a Hopf algebra structure \cite{MR1227098,MR1300632}.
In particular if $V$ and $W$ are $U_q(\Lfg)$-modules, then $V\otimes W$ and $V^*$ are $U_q(\Lfg)$-modules, and
we have $(V \otimes W)^* \cong W^* \otimes V^*$.

\subsection{Finite-dimensional modules}\label{Subsection:QAA}

For a $U_q(\fg)$-module $V$ and $\gl \in P$, set 
\[ V_\gl = \big\{v \in V \mid k_iv = q_i^{\langle \ga_i^\vee, \gl \rangle} v \ \text{for} \ i \in I\big\}.
\]
We say $V$ is \textit{of type $1$} if $V$ satisfies
\[ V = \bigoplus_{\gl \in P} V_{\gl}.
\]
For a finite-dimensional $U_q(\fg)$-module $V$ of type $1$, define its character $\ch V$ by
\[ \ch V = \sum_{\gl \in P} e^{\gl}\dim V_\gl  \in \Z[P].
\]
The category of finite-dimensional $U_q(\fg)$-modules of type $1$ is semisimple.
For $\gl \in P^+$, let $V_q(\gl)$ denote the $U_q(\fg)$-module generated by a nonzero vector $v_\gl$ with relations
\[ x_{i,0}^+v_\gl = 0, \ \ k_iv_\gl = q_i^{\langle \ga_i^\vee, \gl\rangle}v_\gl, \ \ 
  \big(x_{i,0}^-\big)^{\langle \ga_i^\vee, \gl\rangle + 1} v_\gl = 0 \ \ \ \text{for} \ i \in I.
\]
The module $V_q(\gl)$ is simple, finite-dimensional and of type $1$,
and every simple finite-dimensional $U_q(\fg)$-module of type $1$ is isomorphic to some $V_q(\gl)$.
Moreover, we have $\ch V_q(\gl) = \ch V(\gl)$.
For details of these results, see \cite{MR1300632} for example.

Now we recall the basic results on finite-dimensional $U_q(\Lfg)$-modules.
%A $U_q(\Lfg)$-module $V$ is said to be of type $1$ if it is of type $1$ as a $U_q(\fg)$-module.
Let $P_q^+$ denote the monoid (under coordinate-wise multiplication) of $I$-tuples 
of polynomials $\bm{\pi}= \big(\bm{\pi}_1(u),\ldots,\bm{\pi}_n(u)\big)$
such that each $\bm{\pi}_i(u)$ is expressed as
\[ \bm{\pi}_i(u) = (1 -a_{1}u)(1-a_{2}u) \cdots (1-a_{k}u)
\]
for some $k \ge 0$ and $a_{j} \in \C(q)^\times$.
In other words, $P_{q}^+$ is a free abelian monoid generated by $\{ \bm{\varpi}_{i,a} \mid i \in I, a \in \C(q)^{\times}\}$
where 
\[ \big(\bm{\varpi}_{i,a}\big)_{\!j}(u) = \begin{cases} 1 - au & \text{if} \ j = i,\\
                                                    1      & \text{otherwise}.
                                      \end{cases}
\] 
Denote by $P_q$ the corresponding free abelian group, which is called the \textit{$\ell$-weight lattice}.
We say $\bm{\rho} \in P_q$ is \textit{dominant} if $\bm{\rho} \in P_q^+$.
%with coefficients in $\C(q)$ and constant term $1$.
Define a homomorphism $\wt\colon P_q \to P$ by 
\[ \wt(\bm{\varpi}_{i,a}) = \varpi_i \ \ \ \text{for all} \ a \in \C(q)^\times.
\]
A nonzero vector $v$ of a $U_q(\bL\fg)$-module $V$ is said to be an 
\textit{$\ell$-weight vector} with $\ell$-weight $\bm{\rho} \in P_q$ if
there is some $N \in \Z_{> 0}$ satisfying
\[ (\phi_{i,\pm r}^{\pm} - \ggg_{i,\pm r}^{\pm})^N v = 0
\]
for all $i \in I$ and $r \in \Z_{\ge 0}$, 
where $\ggg_{i,\pm r}^{\pm}$ are the rational functions in $q$ determined by the formula
\begin{equation*}
   \sum_{r=0}^{\infty} \ggg_{i,r}^+ u^r = q^{\langle \ga_i^\vee, 
   \wt(\bm{\rho})\rangle}_i\frac{\bm{\rho}_i(q_i^{-1}u)}{\bm{\rho}_i(q_iu)}
   = \sum_{r=0}^{\infty} \ggg_{i,-r}^- u^{-r},
\end{equation*}
in the sense that the left- and right-hand sides are
the Laurent expansions of the middle term about $u=0$ and $u=\infty$, respectively.
Denote by $V_{\bm{\rho}}$ the subspace consisting of $\ell$-weight vectors with $\ell$-weight $\bm{\rho}$.
If $V = \bigoplus_{\bm{\rho} \in P_q} V_{\bm{\rho}}$ holds, we say $V$ is an \textit{$\ell$-weight module}.
For an $\ell$-weight module $V$ and $\mu \in P$, we have
\[ V_{\mu} = \bigoplus_{\begin{smallmatrix} \bm{\rho} \in P_q \\ \wt(\bm{\rho}) = \mu \end{smallmatrix}} V_{\bm{\rho}},
\]
since $\phi_{i,0}^{+} = k_i$ and $\ggg_{i,0}^{+} = q_i^{\langle \ga_i^\vee, \wt(\bm{\rho}) \rangle}$.
In particular, every $\ell$-weight module is of type $1$ as a $U_q(\fg)$-module.
We say a $U_q(\Lfg)$-module $V$ is \textit{$\ell$-highest weight} with $\ell$-highest weight vector $v$ and
$\ell$-highest weight $\bm{\pi} \in P_q^+$ if $v \in V_{\bm{\pi}}$, $x_{i,r}^+v = 0$ ($i \in I, r\in \Z)$
and $U_q(\Lfn_-)v = V$ hold.
A standard argument using (\ref{eq:triangular}) shows that for each $\bm{\pi} \in P_q^+$,
there exists a unique simple $\ell$-highest weight module $L_q(\bm{\pi})$ 
with $\ell$-highest weight $\bm{\pi}$ (up to isomorphism).
By $v_{\bm{\pi}}$ we denote an $\ell$-highest weight vector of $L_q(\bm{\pi})$.

\begin{Thm}[\cite{MR1357195}]
  For every $\bm{\pi} \in P_q^+$, $L_q(\bm{\pi})$ is finite-dimensional and $\ell$-weight.
  Moreover, every simple finite-dimensional $\ell$-weight $U_q(\Lfg)$-module is isomorphic to $L_q(\bm{\pi})$ for some
 $\bm{\pi} \in P_q^+$.
\end{Thm}

\begin{Rem}\normalfont
  For every sequence of (not necessarily splitting) polynomials 
  $\bm{\pi} = \big(\bm{\pi}_1(u),\ldots,\bm{\pi}_n(u)\big)$ such that $\bm{\pi}_i(0) = 1$ ($i \in I$),
  the simple module $L_q(\bm{\pi})$ defined as above is finite-dimensional.
  For later use, however, it is more convenient to restrict our consideration to $L_q(\bm{\pi})$ with $\bm{\pi} \in P_q^+$.
  In particular, this restriction makes it easier to apply the theory of $q$-characters.
  Since $\bm{\pi} \in P_q^+$ implies that $L_q(\bm{\pi})$ is $\ell$-weight by \cite[Theorem 4.1]{MR1810773}, 
  the above theorem follows from \cite{MR1357195}.
%  However, the limitation on $\bm{\pi} \in P_q^+$ does not lose any generalization 
%  Simple modules $L_q(\bm{\pi})$ can be defined for all sequences of (not necessarily splitting) polynomials
%  $\bm{\pi} = \big(\bm{\pi}_1(u),\ldots,\bm{\pi}_n(u)\big)$ such that $\bm{\pi}_i(0) = 1$ ($i \in I$), 
%  and Chari and Pressley proved in \cite{MR1357195} that, such $L_q(\bm{\pi})$'s represent the isomorphism classes 
%  of finite-dimensional $U_q(\bL\fg)$-modules of type $1$.
%  It follows from Proposition \ref{Prop:weight} below that $L_q(\bm{\pi})$ is $\ell$-weight if and only if 
%  $\bm{\pi} \in P_q^+$, and therefore the above theorem is also true.
\end{Rem}

Let $i \mapsto \bar{i}$ be the bijection $I \to I$ determined by $\ga_{\bar{i}} = -w_\circ(\ga_i)$.
We define endomorphisms $\bm{\pi} \mapsto \bm{\pi}^*$ and $\bm{\pi} \mapsto {}^*\!\bm{\pi}$ of the monoid $P_q^+$ by setting
\[ \bm{\varpi}_{i,a}^* = \bm{\varpi}_{\bar{i},aq^{-r^{\vee}h^\vee}}, \ \ \ 
   {}^*\!\bm{\varpi}_{i,a}= \bm{\varpi}_{\bar{i},a^{-1}q^{-r^{\vee}h^{\vee}}}
\]
respectively, where $h^{\vee}$ is the dual Coxeter number of $\fg$ and $r^{\vee} = \max\{c_{ij}c_{ji} \mid i \neq j\}$.
By \cite[Proposition 1.6]{MR1367675}, there is a unique involution $\gs$ of $U_q(\Lfg)$ such that 
\[ \gs(x_{i,r}^\pm) = -x_{i,-r}^{\mp}, \ \ \  \gs(h_{i,m}) = -h_{i,-m}, \ \ \ \gs(k_i^{\pm 1}) = k_i^{\mp 1},\ \ \
   \gs(\phi_{i,r}^{\pm}) = \phi_{i,-r}^{\mp}
\] 
for $i \in I, r \in \Z$ and $m \in \Z\setminus \{ 0\}$, and $\gs$ is also a coalgebra anti-involution.
For a $U_q(\Lfg)$-module $V$, denote by $\gs^*V$ its pull-back with respect to $\gs$.
As a consequence of \cite[Corollary 6.9]{MR1810773}, we have the following lemma.

\begin{Lem} \label{Lem:dual}
  For every $\bm{\pi} \in P_q^+$, we have
  \[ L_q(\bm{\pi})^* \cong L_q(\bm{\pi}^*) \ \ \ \text{and} \ \ \ \gs^*L_q(\bm{\pi}) \cong L_q({}^*\!\bm{\pi})
  \]
  as $U_q(\Lfg)$-modules.
\end{Lem}

\subsection{Minimal affinizations}\label{subsection:minimal}

Here we recall the definition of minimal affinizations and their classifications when the shape of 
the Dynkin diagram of $\fg$ is a straight line, i.e., $\fg$ is of type $ABCFG$.
%In this type, a minimal affinization of every simple $U_q(\fg)$-module is unique up to equivalence.

\begin{Def}[\cite{MR1367675}] \normalfont Let $\gl \in P^+$. \\
  (i) A simple finite-dimensional $U_q(\Lfg)$-module $L_q(\bm{\pi})$ is said to be an \textit{affinization} of $V_q(\gl)$ if
      $\wt(\bm{\pi}) = \gl$.\\
  (ii) Affinizations $V$ and $W$ of $V_q(\gl)$ are said to be \textit{equivalent}
    if they are isomorphic as $U_q(\fg)$-modules.
    We denote by $[V]$ the equivalence class of $V$.
\end{Def}

If $V$ is an affinization of $V_q(\gl)$, as a $U_q(\fg)$-module we have 
\[ V \cong V_q(\gl) \oplus \bigoplus_{\mu < \gl} V_q(\mu)^{\oplus m_{\mu}(V)}
\]
with some $m_{\mu}(V) \in \Z_{\ge 0}$. 
Let $V$ and $W$ be affinizations of $V_q(\gl)$, and define $m_{\mu}(V), m_{\mu}(W)$ as above.
We write $[V] \le [W]$ if for all $\mu \in P^+$, either of the following holds:
\begin{enumerate}
  \item[(i)] $m_\mu(V) \le m_\mu(W)$, or
  \item[(ii)] there exists some $\nu > \mu$ such that $m_{\nu}(V) < m_{\nu}(W)$. 
\end{enumerate}
Then $\le$ defines a partial ordering on the set of equivalence classes of affinizations of $V_q(\gl)$
\cite[Proposition 3.7]{MR1367675}.

\begin{Def}[\cite{MR1367675}] \normalfont
  We say an affinization $V$ of $V_q(\gl)$ is \textit{minimal} if $[V]$ is minimal in the set of 
  equivalence classes of affinizations of $V_q(\gl)$ with respect to this ordering. 
\end{Def}

For $i \in I$, $a \in \C(q)^\times$ and $m \in \Z_{\ge 0}$, define $\bm{\pi}_{m,a}^{(i)} \in P_q^+$ by
\[ \bm{\pi}_{m,a}^{(i)} = \prod_{k = 1}^{m} \bm{\varpi}_{i, aq_i^{m-2k+1}}.
\]
Note that $\bm{\pi}_{0,a}^{(i)}$ is the unit element of the monoid $P_q^+$ for all $i \in I$ and $a \in \C(q)^\times$. 
For $\fg$ of type $ABCFG$, minimal affinizations are completely classified.

\begin{Thm}[\cite{MR1367675,MR1347873}]\label{Thm:Classification}
  Assume that $\fg$ is of type $ABCFG$. For each $\gl\in P^+$,
  there exists a unique minimal affinization of $V_q(\gl)$ up to equivalence.
  Moreover for $\gl = \sum_{i \in I} \gl_i \varpi_i$,
  $L_q(\bm{\pi})$ is a minimal affinization of $V_q(\gl)$ if and only if $\bm{\pi}$ is in the form
  $\prod_{i \in I} \bm{\pi}_{\gl_i, a_i}^{(i)}$ with $(a_i)_{i \in I}$ satisfying one of the following conditions:
  \begin{enumerate}
    \item[(I)] For all $1 \le i < j \le n$, $a_i/a_j = \prod_{i\le k <j} c_k(\gl)$,
    \item[(II)] For all $1 \le i < j \le n$, $a_i/a_j = \prod_{i \le k < j} c_k(\gl)^{-1}$,
  \end{enumerate} 
  where we set $c_k(\gl) = q^{d_k\gl_k + d_{k+1}\gl_{k+1} + d_k -c_{k,k+1}-1}$.
\end{Thm}

\begin{Rem}\normalfont
   (i) Because of different normalizations in some definitions, the conditions of $a_i$'s  
       are rewritten in a slightly different way from the ones in \cite{MR1367675,MR1347873}. \\
   (ii) 
   The situation is more complicated in type $DE$ because of the existence of a trivalent node,
   and the number of the equivalence classes of minimal affinizations of $V_q(\gl)$ differs depending on $\gl$.
   In this case, the classification has been achieved except for $\gl$ orthogonal to the trivalent node
   (see \cite{MR1376937,MR1402568}).
   We omit the details since we do not consider this case in this article.
\end{Rem}

\begin{Def} \normalfont
  The simple modules $L_q(\bm{\pi}_{m,a}^{(i)})$, which are minimal affinizations of $V_q(m\varpi_i)$,
  are called \textit{Kirillov-Reshetikhin modules}.
  Among them, the ones with $m = 1$ are called \textit{fundamental modules}.
\end{Def}

For a nonempty subset $J \subseteq I$ such that $\fg_J$ is simple, 
denote by $P_{q,J}$ the $\ell$-weight lattice of $U_q(\bL\fg_J)$,
and define a map $P_q \ni \bm{\rho} \mapsto \bm{\rho}_J \in P_{q,J}$ by letting $\bm{\rho}_J$ be the $J$-tuple 
$\big(\bm{\rho}_i(u)\big)_{i \in J}$.
%The proof of the following lemma is in \cite{MR1402568}.

\begin{Lem}[{\cite[Lemma 2.3]{MR1402568}}]
%  Let $J \subseteq I$ be a nonempty subset such that $\fg_J$ is simple.
  For every $\bm{\pi} \in P_q^+$, the $U_q(\bL\fg_J)$-submodule of $L_q(\bm{\pi})$ generated by an $\ell$-highest weight 
  vector $v_{\bm{\pi}}$ is isomorphic to the simple $U_q(\bL\fg_J)$-module with $\ell$-highest weight $\bm{\pi}_J$.  
\end{Lem}

For $\mu = \sum_{i \in I} \mu_i \varpi_i \in P$, write $\mu_J = \sum_{i \in J} \mu_i\varpi_i$.
From this lemma and Theorem \ref{Thm:Classification}, the following corollary is easily proved.

\begin{Cor}\label{Cor:restriction}
  Assume that $\fg$ is of type $ABCFG$.
%  Let $J \subseteq I$ be a nonempty subset such that $\fg_J$ is simple.
  If $L_q(\bm{\pi})$ is a minimal affinization of $V_q(\gl)$, then the $U_q(\bL\fg_J)$-submodule
  of $L_q(\bm{\pi})$ generated by $v_{\bm{\pi}}$ is  a minimal affinization of the 
  simple $U_q(\fg_J)$-module with highest weight $\gl_J$. 
\end{Cor}

In type $A$, the structure of minimal affinizations are much simpler than that of the other types.

\begin{Thm}[{\cite[Section 2]{MR841713}, \cite[Theorem 3.1]{MR1402568}}]\label{Thm;typeA}
  Assume that $\fg$ is of type $A$.
  If $L_q(\bm{\pi})$ is a minimal affinization of $V_q(\gl)$, then 
  $L_q(\bm{\pi})$ is isomorphic to $V_q(\gl)$ as a $U_q(\fg)$-module. 
\end{Thm}

\subsection{Classical limits}

In this subsection we assume that $\fg$ is classical, since some of the results below 
(for example Proposition \ref{Prop:A_lattice}) have been proved under this assumption. 

Let $\mathbf{A}=\C[q,q^{-1}] \subseteq \C(q)$.
An $\bf{A}$-submodule $L$ of a $\C(q)$-vector space $V$ is called an $\bf{A}$\textit{-lattice} if 
$L$ is a free $\bf{A}$-module and $\C(q) \otimes_{\bf{A}} L = V$ holds.
Set $\big(x_{i,r}^{\pm}\big)^{(k)} = \big(x_{i,r}^{\pm}\big)^k/[k]_{q_i}!$ for $i \in I$, $r \in \Z$ and $k \in \Z_{\ge 0}$,
and denote by $U_{\mathbf{A}}(\bL\fg)$ the $\mathbf{A}$-subalgebra of $U_q(\bL\fg)$ generated by 
$\big\{ k_i^{\pm 1},  \big(x_{i,r}^{\pm}\big)^{(k)}\bigm| i \in I, r \in \Z, k \in \Z_{> 0} \big\}$.
Define $U_{\mathbf{A}}(\fg)$ in a similar way.
Then $U_{\mathbf{A}}(\Lfg)$ and $U_{\mathbf{A}}(\fg)$ are $\bf{A}$-lattices of $U_q(\Lfg)$ and $U_q(\fg)$ respectively
\cite[Lemma 2.1]{MR1836791}, \cite{MR1227098}.
%\begin{Prop} \label{Prop:A_lattice}
%  We have $U_q(\bL \fg) \cong \C(q) \otimes_{\mathbf{A}} U_{\mathbf{A}}(\bL \fg)$
%  and $U_q(\fg) \cong \C(q) \otimes_{\mathbf{A}} U_{\mathbf{A}}(\fg)$. \\
%\end{Prop}
Set
\[ U_1(\Lfg) = \C \otimes_{\bf{A}} U_{\bf{A}}(\Lfg) \ \ \text{and} \ \ U_1(\fg) = \C \otimes_{\bf{A}} U_{\bf{A}}(\fg),
\]
where $\C$ is regarded as an $\bf{A}$-module by letting $q$ act by $1$.
As shown in the proof of \cite[Lemma 2.1]{MR1836791},
the $\bf{A}$-lattice of the quantum affine algebra $U_q'(\hat{\fg})$ 
in \cite{MR1227098} is mapped onto $U_{\bf{A}}(\Lfg)$ under the canonical projection. 
Hence the following proposition is proved from \cite{MR1227098}, \cite[Proposition 9.3.10]{MR1300632}.

\begin{Prop}\label{Prop:classical_limit}
   The universal enveloping algebra $U(\Lfg)$ is isomorphic to the quotient of $U_1(\Lfg)$ by the ideal generated by 
   $1 \otimes k_i-1\otimes 1$ {\normalfont(}$i \in I${\normalfont)}. 
   In particular if $V$ is a $U_1(\Lfg)$-module on which $1\otimes k_i$'s act by $1$, then $V$ is an $\Lfg$-module.
   A similar statement also holds for $U(\fg)$ and $U_1(\fg)$.
\end{Prop}

Let $P_{\mathbf{A}}^+$ be the submonoid of $P_q^+$ generated by $\{ \bm{\varpi}_{i,a} \mid i \in I, a \in \C^\times q^\Z\}$.
%Following \cite{MR1850556}, we call $\bm{\pi} \in P_q^+$ \textit{integral} if, for all $i \in I$,
%the polynomial $\bm{\pi}_i(u)$ has coefficients in $\bf{A}$ and the coefficient of the highest power of $u$ 
%lies in $\C^\times q^{\Z}$. 
%Denote by $P_{\mathbf{A}}^+$ the set of integral elements in $P_q^+$.
For $\bm{\pi} \in P_{\mathbf{A}}^+$, we set 
$L_{\mathbf{A}}(\bm{\pi}) = U_{\mathbf{A}}(\Lfg)v_{\bm{\pi}} \subseteq L_q(\bm{\pi})$.

\begin{Prop}[\cite{MR1850556,MR1836791}]\label{Prop:A_lattice} \ \\
  {\normalfont(i)} $L_{\mathbf{A}}(\bm{\pi})$ is spanned by the vectors 
      \[ \big(x_{i_1,l_1}^-\big)^{(s_1)} \big(x_{i_2,l_2}^-\big)^{(s_2)} \cdots \big(x_{i_p,l_p}^-\big)^{(s_p)} v_{\bm{\pi}}
      \]
      for $p \ge 0$, $i_j \in I$, $s_j \in \Z_{\ge 0}$ and $0 \le l_j \le N$ with sufficiently large $N$.\\[3pt]
  {\normalfont(ii)} 
      $L_{\mathbf{A}}(\bm{\pi})$ is an $\mathbf{A}$-lattice of $L_q(\bm{\pi})$.
\end{Prop}

Set 
\[ \overline{L_q(\bm{\pi})} = \C \otimes_{\mathbf{A}} L_{\mathbf{A}}(\bm{\pi}),
\]
which is called the \textit{classical limit} of $L_q(\bm{\pi})$.
$\ol{L_q(\bm{\pi})}$ is an $\Lfg$-module by Proposition \ref{Prop:classical_limit}, and we have
\begin{equation}\label{eq:eq_ch}
  \ch L_q(\bm{\pi}) = \ch \ol{L_q(\bm{\pi})}.
\end{equation}
Since $\ch V_q(\mu) = \ch V(\mu)$ holds for all $\mu \in P^+$, we also have 
\begin{equation}\label{eq:eq_multi}
  [L_q(\bm{\pi}): V_q(\mu)] = [\ol{L_q(\bm{\pi})}: V(\mu)]
\end{equation}
for all $\mu \in P^+$, where the left- and right-hand sides are the multiplicities as a $U_q(\fg)$-module and $\fg$-module respectively.

\section{Main theorems and corollaries}\label{Section:Main}

In this section, we assume that $\fg$ is of type $ABC$.

\subsection{Graded limit}

Let $\gl\in P^+$, and let 
$\bm{\pi} \in P_{\mathbf{A}}^+$ be such that $L_q(\bm{\pi})$ is a minimal affinization of $V_q(\gl)$
(such an element exists by Theorem \ref{Thm:Classification}).
Denote by $\ol{\bm{\pi}} = \big(\ol{\bm{\pi}}_1(u),\ldots,\ol{\bm{\pi}}_n(u)\big)$ the $I$-tuple of polynomials 
with coefficients in $\C$ obtained from $\bm{\pi}$ by evaluating $q$ at $1$. 
From Theorem \ref{Thm:Classification}, we easily see that there exists a unique nonzero complex number $a$ satisfying
\[ \ol{\bm{\pi}}_i(u) = (1 -au)^{\langle \ga_i^\vee, \gl \rangle} \ \ \ \text{for all} \ i \in I.
\]
Hence the following lemma is proved from \cite[Lemma 4.7]{MR1850556}.

\begin{Lem}\label{Lem:surj}
  There exists a surjective $\Lfg$-module homomorphism from the classical limit $\ol{L_q(\bm{\pi})}$ 
  to the evaluation module $V(\gl, a)$.
\end{Lem}

Define a Lie algebra automorphism $\tau_a\colon \fg[t] \to \fg[t]$ by 
\[ \tau_a\big(x \otimes f(t)\big) = x \otimes f(t - a) \ \ \ \text{for} \ x \in \fg, f \in \C[t].
\]
We consider $\ol{L_q(\bm{\pi})}$ as a $\fg[t]$-module by restriction, and
define a $\fg[t]$-module $L(\bm{\pi})$ by the pull-back $\tau_a^*\ol{L_q(\bm{\pi})}$.

\begin{Def}\normalfont
  We call the $\fg[t]$-module $L(\bm{\pi})$ the \textit{graded limit} of the minimal affinization $L_q(\bm{\pi})$. 
\end{Def}

In fact, $L(\bm{\pi})$ turns out to be a graded $\fg[t]$-module from our main theorems,
which justifies the name ``graded limit''.
We see from Proposition \ref{Prop:A_lattice} (i) and the construction that 
the vector $\ol{v}_{\bm{\pi}} = 1 \otimes v_{\bm{\pi}}$ generates $L(\bm{\pi})$ as a $\fg[t]$-module. 
Elementary properties of $L(\bm{\pi})$ are as follows.

\begin{Lem}\label{Lem:elementary_minimal}
  {\normalfont (i)} There exists a surjective $\fg[t]$-module homomorphism from $L(\bm{\pi})$ to $V(\gl,0)$. \\
  {\normalfont (ii)} The vector $\ol{v}_{\bm{\pi}}$ satisfies the relations
    \begin{align*}
       \fn_+[t]\ol{v}_{\bm{\pi}} = 0&, \ \ (h\otimes t^s)\ol{v}_{\bm{\pi}} = \gd_{s0}\langle h, \gl \rangle \ol{v}_{\bm{\pi}} 
       \ \ \text{for} \ h \in \fh, s \ge 0, \ \ \text{and} \\
       &f_i^{\langle\ga_i^\vee, \gl\rangle+1}\ol{v}_{\bm{\pi}} = 0 \ \ \text{for} \ i \in I.
    \end{align*}
  {\normalfont (iii)} We have
    \[ \ch L_q(\bm{\pi}) = \ch L(\bm{\pi}).
    \]
  {\normalfont (iv)} For every $\mu \in P^+$, we have
    \[ [L_q(\bm{\pi}) : V_q(\mu)] = [L(\bm{\pi}): V(\mu)].
    \]
%    holds, where the left- and right-hand sides are the multiplicities 
%    as a $U_q(\fg)$-module and $\fg$-module, respectively.
\end{Lem}

\begin{proof}
  The assertion (i) follows from Lemma \ref{Lem:surj},
  and (ii) is proved from the construction of $L(\bm{\pi})$ and (i).
  The assertions (iii) and (iv) are consequences of (\ref{eq:eq_ch}) and (\ref{eq:eq_multi})
  since $\ol{L_q(\bm{\pi})} \cong L(\bm{\pi})$ as $\fg$-modules.
\end{proof}

The following corollary is obvious from Theorem \ref{Thm;typeA} and Lemma \ref{Lem:elementary_minimal}.

\begin{Cor}\label{Prop:typeA}
  When $\fg$ is of type $A$, 
  the graded limit $L(\bm{\pi})$ is isomorphic to the evaluation module $V(\gl, 0)$.
\end{Cor}

\subsection{Main theorems}\label{Main}

Throughout the rest of this section, we fix $\gl = \sum_{i \in I} \gl_i \varpi_i \in P^+$ and
$\bm{\pi} \in P^+_{\mathbf{A}}$ such that $L_q(\bm{\pi})$ is a minimal affinization of $V_q(\gl)$.

In this subsection, we shall state our main theorems.
%which assert that the graded limit of a minimal affinization
%is isomorphic to two modules, one is a generalization
%of a Demazure module introduced in Subsection \ref{Demazure}, and the other is defined in terms of generators and relations
(Although these are trivial in type $A$, we include this type for completeness.)
Their proofs are given in the next section.

Let us define an $n$-tuple $\xi_1,\ldots,\xi_n$ of elements of $\hP$
as follows.
If $\fg$ is of type $A$, then $\xi_i = \gl_i(\varpi_i + \gL_0)$ for all $i \in I$.
If $\fg$ is of type $B$, then
\[ \xi_i = \begin{cases}
             \gl_i(\varpi_i + \gL_0) & \text{if} \ 1 \le i \le n-1, \\
             \gl_n \varpi_n + \lceil \gl_n/2 \rceil \gL_0 & \text{if}\ i = n,
           \end{cases}
\]
where $\lceil s \rceil = \min\{r \in \Z \mid r \ge s\}$.
If $\fg$ is of type $C$, the definition is a little more complicated.
Let $J = \{1 \le i \le n-1 \mid \gl_i >0\} \subseteq I$ and $\hat{J} = \{ 0 \} \sqcup J \subseteq \hI$.
We define $i^{\flat} \in \hI$ for each $i \in I$ by
\[ i^{\flat} = \begin{cases} 
                        \max\{j \in \hat{J} \mid j < i\} & \text{if} \ i \in J,\\
                        i & \text{otherwise}.
          \end{cases}
\]
For $i \in \hat{J}$, we define $p_i \in \{0,1\}$ as follows: if $i = \max\{j \in \hat{J}\}$, then $p_i = 0$.
Otherwise, we define $p_i$ recursively by
\[ p_{i^{\flat}} \equiv \gl_i + p_{i}  \ \ \ \text{mod $2$}.
\]
Set $p_i = 0$ for $i \in I \setminus J$.
Now $\xi_1,\ldots,\xi_n$ is defined by
\[ \xi_i = p_{i^{\flat}} \varpi_{i^{\flat}} + (\gl_i - p_i) \varpi_i + \frac{d_i}{2}\big(\gl_i - p_i + p_{i^{\flat}}\big)\gL_0.
\]
Note that $\sum_{i \in I} \xi_i \in \gl + \Z_{>0} \gL_0$ holds in all cases. 
Our first main theorem is the following.

\begin{Thm}\label{Thm:Main2}
%  Assume $\fg$ is of type $ABC$.
  The graded limit $L(\bm{\pi})$ is isomorphic to $D(w_\circ\xi_1, \ldots, w_\circ\xi_n)$ as a $\fg[t]$-module.
\end{Thm}

In type $A$, $\xi_i\in \hP^+$ holds for all $i$.
Thus the $\fg[t]$-module $D(w_\circ\xi_1,\ldots, w_\circ\xi_n)$ is the submodule of
$V(\gl_1\varpi_1,0) \otimes \cdots \otimes V(\gl_n\varpi_n, 0)$
generated by the tensor product of highest weight vectors, which is isomorphic to $V(\gl,0)$.
Hence the theorem follows from Corollary \ref{Prop:typeA}.

Let $\gD_+^1$ be a subset of $\gD_+$ defined by 
\[ \gD_+^1 = \Big\{ \ga \in \gD_+ \Bigm| \ga = \sum_{i \in I} n_i \ga_i \ \text{with} \ n_i \le 1 \ \text{for all} \ i \in I
   \Big\}.
\]
The second main theorem is the following.

\begin{Thm}\label{Thm:Main1}
  The graded limit $L(\bm{\pi})$ is isomorphic to the cyclic $\fg[t]$-module generated by a nonzero vector
  $v$ with relations
  \begin{align*}\label{eq:rel}
    \fn_+[t]v = 0, \ \ \ &(h \otimes t^s)v= \gd_{s0}\langle h, \gl \rangle v \ \text{for} \ h \in \fh, s \ge 0,\ \ \
    f_i^{\gl_i+1}v = 0 \  \text{for} \ i \in I, \nonumber \\
    &t^2\fn_-[t]v=0 \ \ \text{and} \ \ (f_{\ga}\otimes t)v= 0 \  \text{for} \ \ga \in \gD_+^1.
  \end{align*}
\end{Thm}

When $\fg$ is of type $A$, $\gD_+^1 = \gD_+$ holds and therefore the theorem is easily proved from Corollary \ref{Prop:typeA}.

\begin{Rem}\normalfont
  Theorem \ref{Thm:Main1} implies that $L(\bm{\pi})$ is a projective object in a certain full-subcategory of 
  the category of $\fg \otimes \big(\C[t]/t^2\C[t]\big)$-modules 
  introduced in \cite{MR2763623}.
  In particular this theorem, together with \cite[Theorem 6.1]{MR2342293}, 
  gives a proof to \cite[Conjecture 1.13]{MR2763623} in type $B$.   
\end{Rem}

\begin{Rem}\normalfont
  As stated in the introduction, these two theorems are equivalent to
  \cite[Conjecture 3.20]{MR2587436} in type $AB$.
\end{Rem}

\subsection{Corollaries}\label{Subsection:Corollaries}

Here we shall give several corollaries on minimal affinizations $L_q(\bm{\pi})$,
which are obtained from the corresponding statements on $L(\bm{\pi})$ by applying Lemma \ref{Lem:elementary_minimal} (iii), (iv).

First we apply the results in Subsection \ref{Demazure} to the module 
$D(w_\circ\xi_1,\ldots,w_\circ\xi_n)$ in Theorem \ref{Thm:Main2}.
Let us define  $w_{i} \in \wti{W}$ for each $1 \le i \le n$ as follows:\\
(i) If $\fg$ is of type $A$, then $w_i = \id$ for all $1 \le i \le n$.\\
(ii) If $\fg$ is of type $B$, then
\[ w_i = s_{i-1} s_{i-2} \cdots s_1 \tau,
\]
where $\tau$ denotes the element of $\gS$ which exchanges the nodes $0$ and $1$. \\
(iii) If $\fg$ is of type $C$, then  
\[ w_i = s_{i-1} s_{i-2} \cdots s_1 s_0.
\]
For $1 \le r \le t \le n$, denote by $w_{[r,t]}$ the product $w_{r} w_{r+1} \cdots w_{t} \in \wti{W}$.
If $r > t$, we set $w_{[r,t]} = \id$.

\begin{Lem}\label{Lem:BC}
  {\normalfont(i)} Let $i \in I$.\\
       {\normalfont(a)}  When $\fg$ is of type $B$, we have
         \begin{align*} 
          w_{[1,i]} (\gL_0) &\equiv \varpi_{i} + \gd_{in} \varpi_n + \gL_0 \ \ \text{mod} \ \Q \gd, \ \text{and}\\  
          w_{[1,i]} (\varpi_n + \gL_0) &= \varpi_n + \gL_0.
         \end{align*}
       {\normalfont(b)}  
          When $\fg$ is of type $C$, we have
          \begin{align*}
              w_{[1,i]}(\varpi_j + \gL_0) &\equiv \varpi_{i-j} + \varpi_i + \gL_0 \ \ \text{mod} \ \Q \gd \ \ 
              (0 \le j< i), \  \text{and}\\
              w_{[1,i]}(\varpi_{j} + \gL_0) &= \varpi_j + \gL_0 \ \ (i\le j).
         \end{align*}
  {\normalfont(ii)} %If $\fg$ is of type $B$ {\normalfont(}resp.\ $C${\normalfont)}, 
%    let $n_0 = n$ {\normalfont(}resp.\ $n_0 = n-1${\normalfont)}.
    We have $\ell(w_{[1,n]}) = \sum_{i=1}^{n} \ell(w_i)$.
\end{Lem}

\begin{proof}
    The assertion (i) is proved by direct calculations.
    When $\fg$ is of type $B$, 
   by applying the sequence
   \[ w_{[1,n]} = \tau(s_1 \tau) \cdots (s_{n-1} \cdots s_1 \tau)
   \]
   to $\gL_0$,
    we see that each reflection $s_j$ changes the weight by a positive multiple of 
    $\ga_j$, which implies the assertion (ii).
    The proof for type $C$ is the same.
\end{proof}

Let $\xi_1,\ldots,\xi_n \in \hP$ be the elements defined in the previous subsection,
and set $\gL^i = (w_{[1,i]})^{-1}\xi_i$ for each $i \in I$.
Using Lemma \ref{Lem:BC} (i), the following assertions are easily checked:\\
%These elements (modulo $\Q \gd$) are written as follows:\\
(i) If $\fg$ is of type $A$, then $\gL^i = \xi_i$ for all $i$.\\
(ii) If $\fg$ is of type $B$, then 
    \[ \gL^i \equiv \begin{cases} \gl_i \gL_0 & (1 \le i \le n-1) \\
                             \ol{\gl}_n \varpi_n + \lceil \gl_n/2 \rceil \gL_0 & (i = n)  
               \end{cases}
       \ \ \ \text{mod} \ \Q \gd,
    \]
    where $\ol{\gl}_n = 0$ if $\gl_n$ is even, and $\ol{\gl}_n = 1$ otherwise. \\
(iii) If $\fg$ is of type $C$, then
     \[ \gL^i \equiv \begin{cases} p_{i^\flat} \varpi_{i - i^\flat} + \frac{1}{2}\big(\gl_i - p_i + p_{i^\flat}\big) \gL_0 
                                                      & (1 \le i \le n-1)\\
                              \gl_n (\varpi_n+\gL_0) ( = \xi_n) & (i = n)
                \end{cases}
        \ \ \ \text{mod} \ \Q\gd.
     \]
In particular, $\gL^i \in \hP^+$ for all $i \in I$.
%We see from Lemma \ref{Lem:BC} (i) that
%\begin{equation}\label{eq:isom4}
%  \xi_i = w_{[1,i]} \gL^i \ \ \ \text{for all} \ i \in I.
%\end{equation}
Since $w_\circ\xi_i = w_\circ w_{[1,i]}\gL^i$, we have from Theorem \ref{Thm:Main2} and Proposition \ref{Prop:character2} that
\begin{align}\label{eq:MAisom}
   L(\bm{\pi})& \cong D(w_\circ \xi_1,\ldots,w_\circ\xi_n) \\
   &= F_{w_\circ w_1}\Big(D(\gL^1) \otimes F_{w_2}\Big(D(\gL^2) \otimes \cdots \otimes 
   F_{w_{n-1}}\Big( D(\gL^{n-1}) \otimes F_{w_n} D(\gL^n)\Big)\! \cdots \!\Big)\!\Big).\nonumber
\end{align}
%Then since the isomorphism
%\[ L(\bm{\pi}) \cong D(w_\circ\xi_1,\ldots,w_\circ\xi_n)= F_{w_\circ}D(\xi_1,\ldots,\xi_n)
%\]
%follows from Theorem \ref{Thm:Main2} and Lemma \ref{Lem:one_change}, 
%we see that $L(\bm{\pi})$ is isomorphic to the $\fg[t]$-module
%\begin{equation}
%   F_{w_\circ w_1}\Big(D(\gL^1) \otimes F_{w_2}\Big(D(\gL^2) \otimes \cdots \otimes 
%   F_{w_{n-1}}\Big( D(\gL^{n-1}) \otimes F_{w_n} D(\gL^n)\Big)\! \cdots \!\Big)\!\Big).
%\end{equation}
Now the following character formula for $L_q(\bm{\pi})$
is obtained using Corollary \ref{Prop:character3} and Lemma \ref{Lem:elementary_minimal} (iii).

\begin{Cor}\label{Cor:Cor2}
  \[ \ch L_q(\bm{\pi}) = \D_{w_\circ w_{1}}\Big(e^{\gL^1}\cdot \D_{w_2}\Big(e^{\gL^2} \cdots 
     \D_{w_{n-1}}\Big(e^{\gL^{n-1}}\cdot \D_{w_{n}}(e^{\gL^n})\Big)\!\cdots\!\Big)\!\Big)\Big|_{e^{\gL_0} =e^\gd= 1}.
  \]
\end{Cor}

The next corollary gives multiplicities of $U_q(\fg)$-modules in $L_q(\bm{\pi})$ in terms of crystal bases
(for the basics of crystal bases, see \cite{MR1881971}).
To state it we need to prepare several notation.
Let $B(\gL)$ be the crystal basis of $\hV(\gL)$ for $\gL\in \hP^+$, and $u_\gL$ its highest weight element.
For $\tau \in \gS$, let $b \mapsto b^{\tau}$ denote the bijection from $B(\gL)$ to $B(\tau\gL)$ 
satisfying
\[ u_\gL^\tau = u_{\tau \gL} \ \ \text{and} \ \ \tilde{f}_{\tau(i)} (b^\tau)= (\tilde{f}_ib)^\tau 
   \ \ \text{for} \ i \in \hI,\ b \in B(\gL),
\]
where $\tilde{f}_i$ are the Kashiwara operators.
Let $\gL_{(1)},\ldots, \gL_{(p)}$ be an arbitrary sequence of elements of $\hP^+$.
For a subset $T$ of the crystal basis
$B\big(\gL_{(1)}\big) \otimes \cdots \otimes B\big(\gL_{(p)}\big)$,
we define a subset $\cF_\tau T$ by
\[ \cF_\tau T = \big\{b_1^{\tau} \otimes \cdots \otimes b_p^{\tau} \bigm| b_1\otimes \cdots \otimes b_p \in T\big\} 
   \subseteq B\big(\tau \gL_{(1)}\big) \otimes \cdots \otimes B\big(\tau \gL_{(p)}\big).
\]
For $w \in \hW$ with reduced expression $w = s_{i_1} \cdots s_{i_k}$, we also define a subset $\cF_w T$ by
\[ \cF_w T = \big\{\tilde{f}_{i_1}^{s_1} \tilde{f}_{i_2}^{s_2} \cdots \tilde{f}_{i_k}^{s_k} (b) \bigm| 
    s_j \ge 0, b \in T\big\} \setminus \{ 0\} \subseteq B\big(\gL_{(1)}\big) \otimes \cdots \otimes B\big( \gL_{(p)}\big),
\]
and set $\cF_{w\tau } = \cF_{w}\cF_{\tau}$.
Now let us define a subset $Z'$ of a $U_q(\hfg)$-crystal basis by
\begin{align*}
      Z' = &\cF_{w_\circ w_1}\Big(u_{\gL^1} \otimes \cF_{w_2}
      \Big(u_{\gL^2} \otimes\cdots \otimes \cF_{w_{n-1}}\Big(u_{\gL^{n-1}} \otimes
      \cF_{w_n}(u_{\gL^n})\Big) \!\cdots\! \Big)\!\Big).
\end{align*}
Since this is a crystal analogue of the right-hand side of (\ref{eq:MAisom})
(see \cite{MR1887117}, in which $B(\gL)$ are realized using LS paths),
the classically highest weight elements in $Z'$ 
(i.e., elements annihilated by $\tilde{e}_i$ for $i \in I$) are in one-to-one correspondence
with the simple $\fg$-module components of $L(\bm{\pi})$.
Note that $\cF_{w_\circ}$ generates no new classically highest weight elements since $w_\circ \in W$.
Hence the same statement also holds for 
\begin{align*}
   Z = \cF_{w_1}\Big(u_{\gL^1} \otimes \cF_{w_2} \Big( u_\gL^2 \otimes \cdots \otimes \cF_{w_{n-1}}\Big(u_{\gL^{n-1}} \otimes
      \cF_{w_n}(u_{\gL^n})\Big) \!\cdots\! \Big) \!\Big)
\end{align*}
instead of $Z'$.
From this and Lemma \ref{Lem:elementary_minimal} (iv), we have the following corollary.

\begin{Cor}\label{Cor:crystal}
  For every $\mu \in P^+$, we have
  \begin{equation*}\label{eq:Cor2}
     \big[L_q(\bm{\pi}): V_q(\mu)\big]= \# \big\{ b \in Z \bigm| \text{$\fh$-weight of $b$ is $\mu$},\
                                                  \tilde{e}_i(b) = 0 \ \text{for} \ i \in I \big\}.
  \end{equation*}
\end{Cor}

\begin{Rem}\label{Rem}\normalfont
  Assume that $\fg$ is of type $BC$, and let $J = \{0,1,\ldots,n-1\} \subseteq \hI$, 
  $U_q(\hfg_J)$ be the subalgebra of $U_q(\hfg)$ whose set of simple roots are $J$, and $W_J$ its Weyl group.
  Since $w_i$ belongs to $W_J \rtimes \gS$ for all $i \in I$, 
  we can regard $Z$ as a subset of a $U_q(\hfg_J)$-crystal basis.
  In view of this, Corollary \ref{Cor:crystal} implies that multiplicities of $L_q(\bm{\pi})$ are expressed
  in terms of crystal bases of \textit{finite type} ($D_n$ and $C_n$ respectively).
\end{Rem}

Finally we prove the formula for the limit of normalized characters of minimal affinizations, which has been conjectured in 
\cite[Conjecture 6.3]{MY}.
Let $J$ be a subset of $I$, and set
$\gD_+^J = \gD_+ \cap\Big( \sum_{ i \in J} \Z \ga_i\Big)$
and 
\[ \gD_+^{1,J} = \Big\{ \ga \in \gD_+ \Bigm| \ga = \sum_{i \in I} n_i\ga_i \ \text{with} \ n_i \le 1 \ \text{if} \ i \notin J
   \Big\}.
\] 
Assume that $\gl^1, \gl^2, \ldots$ is an infinite sequence of elements of $P^+$ such that
\[ \langle \ga_i^\vee, \gl^k \rangle = 0 \ \ \text{for all} \ i \in J,\ k\in \Z_{>0} \ \ \text{and} \ \ 
   \lim_{k \to \infty} \langle \ga_i^\vee, \gl^k \rangle = \infty \ \ \text{for all} \ i \notin J.
\]

\begin{Cor}\label{Cor:Cor3}
  Let $\bm{\pi}^1,\bm{\pi}^2,\ldots$ be an infinite sequence of elements of $P^+_{\mathbf{A}}$ 
  such that $L_q(\bm{\pi}^k)$ is a minimal affinization of $V_q(\gl^k)$.
  Then $\lim_{k \to \infty} e^{-\gl^k}\mathrm{ch}\, L_q(\bm{\pi}^k)$ exists, and
  \begin{equation}\label{eq:limit}
     \lim_{k \to \infty} e^{-\gl^k}\mathrm{ch} \, L_q(\bm{\pi}^k) 
     = \prod_{\ga \in \gD_+\setminus \gD_+^J} \frac{1}{1 - e^{-\ga}} \cdot 
       \prod_{\ga \in \gD_+ \setminus \gD_+^{1,J}} \frac{1}{1 -e^{-\ga}}.
  \end{equation}  
\end{Cor}

\begin{proof}
  By Lemma \ref{Lem:elementary_minimal} (iii), 
  it suffices to show that $\lim_{k \to \infty} e^{-\gl^k}\chh L(\bm{\pi}^k)$ coincides with the right-hand side
  of (\ref{eq:limit}).
  Define a Lie subalgebra $\fa_J$ of $\fg[t]$ by 
  \[ \fa_J= \fn_+[t] \oplus \fh[t] \oplus \bigoplus_{\ga \in \gD_+^J} \C f_\ga \oplus \bigoplus_{\ga \in \gD_+^{1,J}} 
     \C (f_\ga \otimes t) \oplus t^2\fn_- [t].
  \]
  For each $\gl^k$, let $\C v^k$ be a $1$-dimensional $\fa_J$-module defined by
  \[  hv^k = \langle h, \gl^k \rangle v^k \ \ \text{for} \ h \in \fh, \ \ \  (\fa_J \cap \hfn_+)v^k 
      = f_\ga v^k =0 \ \ \text{for} \ \ga \in \gD_+^{J},
  \]
  and define a $\fg[t]$-module $M^k$ by
  \[ M^k =U(\fg[t]) \otimes_{U(\fa_J)} \C v^k.
  \]
  For all $k$, $e^{-\gl^k} \mathrm{ch}_{\fh}\, M^k$ coincides with the right-hand side of (\ref{eq:limit}).
  Note that both $M^k$ and $L(\bm{\pi}^k)$ have natural $\Z_{\ge 0}$-gradings,
  which we normalize so that the degrees of $v^k$ and $\ol{v}_{\bm{\pi}^k}$ are $0$.
  We denote these gradings by superscripts.
  By Theorem \ref{Thm:Main1}, there exists a surjective homomorphism $\Phi^k \colon M^k \to L(\bm{\pi}^k)$, and 
  $\ker \Phi^k$ is generated by the vectors $w^k_i = f_i^{\langle \ga_i^\vee, \gl^k \rangle +1}v^k$ for $i \notin J$.
  A standard calculation shows $\fn_+w^k_i = 0$, which implies $\ker \Phi^k = \sum_{i\notin J}
  U(\fn_-[t])U(t\fh[t])U(t\fn_+[t])w^k_i$.
  Hence if $\gb \in Q^+$ and $s \ge 0$ satisfies $(\ker \Phi^k)_{\gl^k - \gb}^s \neq 0$, there exists some $i \notin J$
  satisfying
  \[  \gl^k - \gb + s\gd\in \gl^k - \big(\langle \ga_i^\vee, \gl^k \rangle +1\big)\ga_i + w_\circ\hQ^+.
  \]
  When $\gb$ and $s$ are fixed, this does not occur for sufficiently large $k$, which implies 
  \[ \dim (M^k)_{\gl^k - \gb}^s = \dim L(\bm{\pi}^k)_{\gl^k - \gb}^s \ \ \ \text{if} \ k \gg 0.
  \]
  Since $(M^k)_{\gl^k -\gb}^s = 0$ except for finitely many $s$ if $\gb$ is fixed, the assertion follows.
\end{proof}

\subsection{Example}\label{Subsection:Example}

In order to illustrate how Corollary \ref{Cor:crystal} is applied,
we will give an explicit decomposition for minimal affinizations of type $B$ in a special case.
Results in this subsection have previously been given in \cite{MR2587436} in a different way.

Assume that $\fg$ is of type $B_n$.
For convenience we relabel (in this subsection only) the vertices of the Dynkin diagram of $\hfg$
as follows:

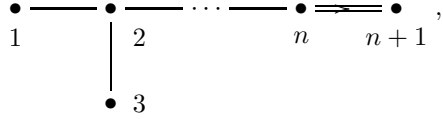
\[ \xygraph{!~:{@{=}|@{>}}
    \bullet ([]!{+(0,-.3)} {1}) - [r]
    \bullet ([]!{+(.3,-.3)} {2}) (
        - [d] \bullet ([]!{+(.3,0)} {3}),
         - [r] \cdots - [r] \bullet ([]!{+(0,-.3)} {n})
         : [r] \bullet ([]!{+(0,-.3)} {n+1})
    )},
\]
and we let $\fg$ be the Lie subalgebra corresponding to the subdiagram with vertices $I = \{1,2,4,\ldots,n+1\}$.
In the rest of this subsection we use the new labels for $\varpi_i$, $\ga_i$, etc.,
and consider the case where $\gl = \sum_{i \in I}m_i \varpi_i \in P^+$ with $m_i = 0$ for $i >4$,
that is, $\gl = m_1\varpi_1 + m_2 \varpi_2 + m_4 \varpi_4$.

In this setting, the subset $Z$ defined in the previous subsection becomes
\[ Z = u_{m_1\gL_1} \otimes \mathcal{F}_{s_3}\Big(u_{m_2\gL_3} \otimes \mathcal{F}_{s_2s_1}(u_{m_4 \gL_1})\Big),
\]
where $\gL_i$ denote the fundamental weights of $\hfg$ (with respect to the new labels).
Let $J = \{1,2,3\}$ and $U_q(\hfg_J)$ be the subalgebra of $U_q(\hfg)$ of type $A_3$  defined similarly as in Remark \ref{Rem}.
Then as in the remark we can regard $Z$ as a subset of the $U_q(\hfg_J)$-crystal basis
$B_J(m_1\gL_1) \otimes B_J(m_2\gL_3) \otimes B_J(m_4\gL_1)$,
where $B_J(\nu)$ denotes the crystal basis of the simple $U_q(\hfg_J)$-module with highest weight $\nu$.
In the sequel, we regard $Z$ in this way.

Define a subset $W \subseteq B_J(m_2\gL_3) \otimes B_J(m_4\gL_1)$ by
\[ W = \{b \in B_J(m_2\gL_3) \otimes B_J(m_4\gL_1)\mid \tilde{e}_1^{m_1+1}(b) = 0, \ \tilde{e}_2(b) = 0\}.
\]
Then by the tensor product rule (see \cite{MR1881971}), we have
\begin{equation}\label{eq:example}
   \{b \in Z \mid \tilde{e}_i(b) = 0 \ \text{for} \ i =1,2\} = u_{m_1\gL_1} \otimes W.
\end{equation}
As \cite[Theorem 8.2.1]{MR1881971}, let us  identify $B_J(m_2\gL_3)$ and $B_J(m_4\gL_1)$ with the sets of 
semistandard tableaux of rectangle shape $3 \times m_2$ and $1 \times m_4$ respectively.
Then it is easily checked that $W$ is the set consisting of the elements of the form

\begin{align*}
   &\fbox{$1\ \ \ \ \ \cdots \hspace{14.7pt} 1$}  \\[-5pt]
    &\fbox{$2 \ \ \ \ \ \cdots \hspace{14.7pt} 2$}  \otimes \fbox{$1\cdots 1$}\hspace{-2pt}
    \underbrace{\fbox{$2\cdots 2$}}_y
   \hspace{-2pt}\underbrace{\fbox{$4\cdots4$}}_z \\[-20pt]
    & \fbox{$3\cdots3$}\hspace{-2pt}\underbrace{\fbox{$4\cdots4$}}_x
\end{align*}
with $x,y,z \ge 0$, $x \le m_2$, $y \le m_1$, and $y + z \le m_4$.
Note that the weight of this element is 
\[ m_2 \gL_3 + m_4\gL_1 - (y+z)\ga_1 - z\ga_2 - (x+z) \ga_3.
\]
Now by Corollary \ref{Cor:crystal} and (\ref{eq:example}), we have that a minimal affinization $L_q(\bm{\pi})$ of $V_q(\gl)$ 
decomposes as follows:
\[ L_q(\bm{\pi}) \cong \bigoplus_{\begin{smallmatrix} x,y,z \ge 0, \\ x \le m_2, y \le m_1, y+z \le m_4\end{smallmatrix}} 
    V_q\Big((m_1+m_4) \varpi_1 - (y+z)\ga_1 - z\ga_2 - (x+z) \ol{\ga}_3\Big),
\]
where $\ol{\ga}_3$ denotes the restriction of $\ga_3$ on $\fh$ (note that $\gL_3|_{\fh}= 0$).
It is easily checked that this result coincides with that given in \cite[Proposition 5.7]{MR2587436}.

\section{Proof of main theorems}

Throughout this section, we assume that $\fg$ is of type $BC$ unless specified otherwise,
and fix $\gl = \sum_{i \in I} \gl_i \varpi_i \in P^+$ and $\bm{\pi} \in P_{\mathbf{A}}^+$ such that 
$L_q(\bm{\pi})$ is a minimal affinization of $V_q(\gl)$.
Here we use the usual labeling given in Section \ref{Section:Lie}, not the one in Subsection \ref{Subsection:Example}. 
We freely use the notation $\xi_i$, $\gL^i$, etc., defined in Subsections \ref{Main} and \ref{Subsection:Corollaries}.

Let $M(\gl)$ denote the $\fg[t]$-module defined in terms of generators and relations in Theorem \ref{Thm:Main1}.
We shall verify one by one the existence of surjective homomorphisms
\[ D(w_\circ\xi_1, \ldots,w_\circ\xi_n) \twoheadrightarrow M(\gl), \ \ \ M(\gl) \twoheadrightarrow L(\bm{\pi}), \ \ \ 
   L(\bm{\pi}) \twoheadrightarrow D(w_\circ\xi_1, \ldots,w_\circ\xi_n),
\]
which proves Theorems \ref{Thm:Main2} and  \ref{Thm:Main1} simultaneously.
For the proof of the latter two, we need some results on $q$-characters. 
These are recalled in Subsection \ref{Subsection:q-char}.

\subsection{Proof of \boldmath$D(w_\circ\xi_1, \ldots,w_\circ\xi_n) \twoheadrightarrow M(\gl)$}\label{Proof1}

Let us prepare several notation.
For $1 \le p \le q \le n$, set $\ga_{p,q} = \ga_p + \ga_{p+1} + \cdots + \ga_q$.
We have
\[ \gD_+ = \begin{cases} \{\ga_{p,q} \mid p\le q\} \sqcup \{\ga_{p,n} + \ga_{q,n} \mid p<q\} & \fg = B_n,\\
                         \{\ga_{p,q} \mid p \le q \} \sqcup \{\ga_{p,n} + \ga_{q,n-1} \mid p \le q <n \} & \fg =C_n.
           \end{cases}
\]
Set $\ga_{p,q} = 0 $ if $p > q$.
Let $v_M$ denote the generator of $M(\gl)$ in the definition,
and set $D = D(w_\circ\xi_1, \ldots,w_\circ\xi_n)$ and $v_D = v_{\xi_1} \otimes \cdots \otimes v_{\xi_n} \in D$.
Note that $v_D$ generates $D$ as a $\fg[t]$-module.
Recall that $D$ is by definition a module over the Lie algebra $\fg[t] \oplus \C K \oplus \C d$,
and $K$ and $d$ act on $v_D$ by some scalar multiplications.
In this subsection, we also view $M(\gl)$ as a module over this Lie algebra by letting 
$K$ and $d$ act on $v_M$ by the same multiplications. 

For $\ga = \gb + s\gd \in \hat{\gD}^{\mathrm{re}}$, we define an nonnegative integer $\rho(\ga)$ by
\[ \rho(\ga) = \sum_{k=1}^n \max\big\{0, -\langle \ga^\vee,\xi_k\rangle\big\}
   = \sum_{k=1}^n \max \Big\{0, - \langle \gb^\vee, \xi_k\rangle - \frac{2s}{(\gb,\gb)} \langle K, \xi_k \rangle \Big\}. 
\] 
The following assertions are checked by direct calculations.\\[6pt]
    {\normalfont(i)} For $\gb + s \gd \in \hgDre_+$, we have 
       $\rho(\gb+s\gd) = 0$ unless $-\gb \in \gD_+ \setminus \gD_+^1$ and $s = 1$.\\[3pt]
    {\normalfont(ii)} If $\fg$ is of type $B$ and $p < q$, we have
      \[ \rho\big(-(\ga_{p,n} + \ga_{q,n})+ \gd\big)  = \sum_{k = q}^{n-1} \gl_k + \big\lfloor \gl_n/2 \big\rfloor.
      \] 
    {\normalfont(iii)} If $\fg$ is of type $C$ and $p<q$, we have  
    \begin{align*}
      \rho\big(-(\ga_{p,n} &+ \ga_{q,n-1})+\gd\big) \\
                      =&\begin{cases} \sum_{k=q}^{n-1} \gl_k -1 & \text{if} \ \sum_{k=q}^{n-1}\gl_k \in 2\Z_{\ge 0} + 1
                                        \ \text{and} \ 
                                      \gl_k=0 \ \text{for all} \ p\le k <q,\\[3pt]
                                      \sum_{k=q}^{n-1} \gl_k & \text{otherwise},
                        \end{cases}
    \end{align*}
    and $\rho\big(-(\ga_{q,n} + \ga_{q,n-1})+\gd\big) = \Big\lfloor \sum_{k = q}^{n-1} \gl_k/2 \Big\rfloor$.\\  
%  These are proved by direct calculations.
%  For example, let us check the assertion for $\fg = C_n$, $s=1$ and $\ga = -(\ga_{p,n} + \ga_{q,n-1})$ with $p < q <n$.
%  Since $\ga^\vee = -\sum_{l=p}^{q-1} \ga_l^\vee - 2 \sum_{l=q}^n\ga_l^\vee$ and $(\ga,\ga) = 1$, 
%  we have for $1 \le k \le n-1$ that
%  \begin{align*}
%     -\langle (\ga+\gd)^\vee,\xi_k\rangle 
%     &= \Big\langle -\ga^\vee, (m_k -p_k)\varpi_k + p_{e(k)}\varpi_{e(k)}\Big\rangle - (m_k - p_k + p_{e(k)})\\
%     &\begin{cases} =m_k - p_k + p_{e(k)} & \text{if} \ q \le k, \\
%                    =p_{e(k)} & \text{if} \ p\le k< q \ \text{and} \ q \le e(k), \\
%                    \le 0 & \text{otherwise}.
%     \end{cases}
%  \end{align*}  
%  Since $\xi_n \in \hat{P}^+$, we also see $-\langle (\ga+\gd)^\vee, \xi_n \rangle \le 0$.
%  Now the assertion is proved from them.
%  The proofs for the other cases are similar.

For $\ga = \gb + s\gd \in \hgDre$, we denote by $x_\ga$ the vector $e_\gb \otimes t^s \in \hfg$.
The following proposition is essential in this subsection.

\begin{Prop}\label{Prop:annihilator}
  {\normalfont(i)} If $\fg$ is of type $B$, we have 
      \[ \mathrm{Ann}_{U(\hfn_+)}v_D = U(\hfn_+)\bigg(\sum_{\ga \in \hgDre_+} \C x_\ga^{\rho(\ga)+1} + t\fh[t]\bigg).
      \]
  {\normalfont(ii)} If $\fg$ is of type $C$, we have
      \[ \mathrm{Ann}_{U(\hfn_+)}v_D = U(\hfn_+)\bigg(\sum_{\ga \in \hgDre_+} \C x_\ga^{\rho(\ga)+1} +
         \sum_{\begin{smallmatrix} (\ga,\gb) \in S \\ 1\le k < \rho(\ga)/2+1 \end{smallmatrix}}
         \C x_{\ga}^{\rho(\ga)-2k +1} x_{\gb}^k + t\fh[t]\bigg),
      \]
      where $S$ is a subset of $\hat{\gD}_+^{\mathrm{re}} \times \hat{\gD}_+^{\mathrm{re}}$ defined by
       \[ S= \bigg\{ \Big(-\big(\ga_{p,n} + \ga_{q,n-1}\big) + \gd, -\big(\ga_{q,n} + \ga_{q,n-1}\big) + \gd \Big)\biggm|
                     1\le p < q \le n\bigg\}.
       \]
\end{Prop}

%\begin{Rem}\label{Rem}\normalfont
%  Since $T(\gl) \cong D(\xi_1, \ldots,\xi_n)$, it is easy to see from Proposition \ref{Prop:relations} 
%  that $\big\{x_\gb^{\rho(\gb)+1} \mid \gb \in \hat{\gD}_+^{\mathrm{re}}\big\}$ 
%  and $t\fh[t]$ annihilate $v_T$.
%\end{Rem}

Assuming this proposition for a while, we shall prove $D \twoheadrightarrow M(\gl)$ first.
Set $D' = D(\xi_1,\ldots,\xi_p) \subseteq D$, which is generated by $v_D$ as a $\hfb$-module.

\begin{Lem}\label{Lem:exis}
  There is a $\hfb$-module homomorphism from $D'$ to $M(\gl)$ mapping $v_D$ to $v_M$.
\end{Lem}

\begin{proof}
  Since the $\hfh$-weights of $v_D$ and $v_M$ are same,
  it is enough to check that $\mathrm{Ann}_{U(\hfn_+)}v_D$ given in Proposition \ref{Prop:annihilator} annihilates $v_M$.
  First let us show $x_\ga^{\rho(\ga)+1}v_M = 0$ for $\fg=B_n$ and $\ga = -(\ga_{p,n} + \ga_{q,n}) + \gd$ with $p< q$.
  Set $\ggg = \ga_{q,n}$.
  It is easily checked that
  \[ \Big\langle \ggg^\vee, \gl + \big(\rho(\ga)+1\big)\ga \Big\rangle= \bigg(2\sum_{k=q}^{n-1}\gl_k +\gl_n\bigg) -
     2 \bigg( \sum_{k=q}^{n-1}\gl_k + \lfloor \gl_n/2\rfloor +1 \bigg) < 0.
  \]
  On the other hand, a direct calculation shows $x_{\ggg}x_\ga^{\rho(\ga)+1}v_M = 0$,
  which implies $x_\ga^{\rho(\ga)+1}v_M = 0$ as desired since $M(\gl)$ is a finite-dimensional $\fg$-module.
  For $\fg = C_n$ and $\ga = -(\ga_{q,n} + \ga_{q,n-1}) + \gd$, $x_\ga^{\rho(\ga)+1}v_M = 0$ 
  is proved by the same argument with $\ggg = \ga_{q,n-1}$.

  Next we shall show $x_\ga^{\rho(\ga) -2k+ 1}x_\gb^{k}v_M = 0$ for $\fg = C_n$,
  \[ \ga = -(\ga_{p,n} + \ga_{q,n-1}) +\gd, \ \ \gb = -(\ga_{q,n} + \ga_{q,n-1}) + \gd \ \ \text{with} \ p<q,
  \] 
  and $0 \le k< \rho(\ga)/2+1$.
  If $\rho(\ga) = \sum_{k =q}^{n-1}\gl_k$, this is proved by the same argument as above with $\ggg = \ga_{q,n-1}$.
  So we may assume that $\rho(\ga) = \sum_{k =q}^{n-1}\gl_k -1 $ (i.e., $\sum_{k=q}^{n-1}\gl_k$ is odd and $\gl_k = 0$ 
  for $p \le k <q$). 
  The assertion is proved by the descending induction on $k$.
  Set $\ggg = \ga - \gb = -\ga_{p,q-1}$.
  Since $\rho(\gb) = \rho(\ga)/2$ and $x_\ggg v_M= 0$, we have
  \[ 0 = x_\ggg x_\gb^{\rho(\ga)/2+1}v_M \in \C^\times x_{\gb + \ggg} x_\gb^{\rho(\ga)/2}v_M.
  \]
  Hence the case $k = \rho(\ga)/2$ is proved.
  Assume that $k < \rho(\ga)/2$. 
  A direct calculation using $x_\ggg v_M=0$ shows 
  \begin{align*}
     x_\ggg^2 x_{\gb + \ggg}^{\rho(\ga) - 2k -1} x_\gb^{k+2} v_M = 
      a_1 &x_{\gb + 2\ggg}^2 x_{\gb+\ggg}^{\rho(\ga)-2k-3} x_\gb^{k+2} v_M\\
      &+ a_2 x_{\gb + 2\ggg}x_{\gb+\ggg}^{\rho(\ga) -2k -1} x_\gb^{k+1}v_M + a_3 x_{\gb+\ggg}^{\rho(\ga) -2k+ 1}x_\gb^{k}v_M
   \end{align*}
   for some $a_1,a_2,a_3 \in \C$ with $a_3 \neq 0$ (set $x_{\gb + \ggg}^l =0$ if $l < 0$).
   By the induction hypothesis, this implies $x_\ga^{\rho(\ga) -2k+ 1}x_\gb^{k}v_M = 0$ as desired.
   The other relations are trivially checked, and the lemma is proved.
\end{proof}

Let $w_\circ = s_{i_1} \cdots s_{i_{r-1}}s_{i_r}$ be a reduced expression of $w_\circ$, and set $w^{k\le} = s_{i_{k}} 
\cdots s_{i_r}$ for $1 \le k \le r+1$.
We define $D^{k}= F_{w^{k\le}}D'$.
Then $D^{r+1} = D'$ is obvious, and $D^{1} = D$ follows from Lemma \ref{Lem:one_change}.
In the following, we shall verify by the descending 
induction on $k$ that there exists a nonzero $\hfb$-module homomorphism from $D^k$ to $M(\gl)$.
This for $k=r+1$ is just Lemma \ref{Lem:exis}.
Assume that $k \le r$,
and consider a $\hfp_{i_k}$-module $U(\hfp_{i_k}) \otimes_{U(\hfb)} D^{k+1}$.
This $\hfp_{i_k}$-module has a unique maximal finite-dimensional quotient \cite{MR826100},
which we denote by $\wti{D}^{k+1}$.
We easily see from the definition that, if $N$ is a finite-dimensional $\hfp_{i_k}$-module, 
every $\hfb$-module homomorphism $D^{k+1} \to N$
uniquely extends to a $\hfp_{i_k}$-module homomorphism $\wti{D}^{k+1} \to N$. 
%In other words, $\mathfrak{D}_i$ denotes the left adjoint functor of 
%the restriction functor $\hfp_i\mathchar`-\mathrm{Mod}_{\hfh}
%\to \hfb\mathchar`-\mathrm{Mod}_{\hfh}$, where $\hfp_i\mathchar`-\mathrm{Mod}_{\hfh}$ is the 
%category of finite-dimensional $\hfh$-semisimple $\hfp_i$-modules and $\hfb\mathchar`-\mathrm{Mod}_{\hfh}$ is
%defined similarly.
By the induction hypothesis, there is a nonzero $\hfb$-module homomorphism $D^{k+1} \to M(\gl)$,
which extends to a $\hfp_{i_k}$-module homomorphism $\wti{D}^{k+1} \to M(\gl)$.
Hence it suffices to show that $\wti{D}^{k+1} \cong D^k$.
The inclusion $D^{k+1} \hookrightarrow  F_{i_{k}}D^{k+1} = D^k$ extends to a homomorphism $\wti{D}^{k+1} \to D^k$, 
%\[ \wti{D}^{k+1} \to F_{i_k}D^{k+1} = U(\hfp_{i_k})D^{k+1} = D^k,
%\]
and this is obviously surjective.
On the other hand, \cite[Lemmas 2.6, 2.8(i)]{MR826100} and Theorem \ref{Thm:LLM} imply
\[ \chhh \wti{D}^{k+1} = \mathcal{D}_{i_k}\chhh D^{k+1} = \chhh F_{i_k}D^{k+1} = \chhh D^k,
\]
and therefore $\wti{D}^{k+1} \cong D^k$ holds, as desired.

By the above argument, we see that there is a nonzero $\hfb$-module homomorphism from $D$ to $M(\gl)$.
Note that $D$ and $M(\gl)$ are generated by the $1$-dimensional weight spaces
$D_{w_\circ\gl}$ and $M(\gl)_{w_\circ \gl}$ respectively,
and these spaces are annihilated by $\fn_-$.
We easily see from this that the homomorphism $D \to M(\gl)$ is surjective, and extends to one of $\fg[t]$-modules.
Hence $D \twoheadrightarrow M(\gl)$ (as $\fg[t]$-modules) is proved.

It remains to show Proposition \ref{Prop:annihilator}.
For $1 \le j \le n$ and $0 \le i \le j$,
let
\begin{align*}
   D(i,j)&= D\big(w_i\gL^j, w_iw_{j+1}\gL^{j+1}, \ldots,w_iw_{[j+1,n-1]}\gL^{n-1},w_i w_{[j+1,n]}\gL^n\big), \ \text{and}\\
   v(i,j)& = v_{w_i\gL^j} \otimes v_{w_iw_{j+1}\gL^{j+1}} \otimes \cdots \otimes v_{w_iw_{[j+1,n-1]}\gL^{n-1}}
   \otimes v_{w_i w_{[j+1,n]}\gL^n}
\end{align*}
(here we set $w_0 = \id$). The vector $v(i,j)$ is a generator of $D(i,j)$ as a $\hfb$-module.
Note that $D' = D(1,1)$ and $v_D = v(1,1)$ since $\xi_i = w_{[1,i]}\gL^i$ for all $i \in I$.
By Lemmas \ref{Lem:one_change} and \ref{Lem:tau}, we have 
\begin{align*}\label{eq:induction}
  F_i' D(i,j) = D(i+1,j) \ \ \ \text{for}\ 0\le i<j,
\end{align*}
where $F_i' = F_\tau$ if $\fg = B_n$ and $i = 0$, and $F_i' = F_i$ otherwise.
Moreover $D(0,j) = D(\gL^j) \otimes D(j+1,j+1)$ holds for $j < n$.
In the following, we shall prove the proposition by determining the annihilators of $v(i,j)$'s inductively 
using these equalities.
For this, we prepare two lemmas. 
Define a Lie subalgebra $\hfn_i$ of $\hfn_+$ for $i \in \hI$ by $\hfn_i = 
\bigoplus_{\ga \in \hat\gD_+^{\mathrm{re}} \setminus \{\ga_i\}} \C x_\ga \oplus t\fh[t]$.
Note that $\hfn_+ = \C e_i \oplus \hfn_i$.
%The following technical lemma is proved by the argument in the proof of \cite[Theorem 3.4]{MR826100}.
%To do this, we need the following technical lemma.

\begin{Lem}\label{Lem:proceed}
  Let $V$ be an integrable $\hfg$-module, $T$ a finite-dimensional $\hfb$-submodule of $V$,
  $i \in \hI$ and $\xi \in \hP$ such that $\langle \ga_i^\vee, \xi \rangle \ge 0$.
  Assume that the following conditions hold:\\
  {\normalfont(i)}
    $T$ is generated by a weight vector $v \in T_\xi$ satisfying $e_i v=0$.\\
  {\normalfont(ii)}
    There is an $\ad(e_i)$-invariant left $U(\hfn_i)$-ideal $\mathcal{I}$ such that 
    \[ \mathrm{Ann}_{U(\hfn_+)}v = U(\hfn_+)e_i + U(\hfn_+)\mathcal{I}.
    \]
  {\normalfont(iii)} We have $\chhh F_i T = \mathcal{D}_i \chhh T$.\\
  Let $v' = f_i^{\langle \ga_i^\vee, \xi \rangle} v$. Then we have
  \[ \mathrm{Ann}_{U(\hfn_+)} v'= U(\hfn_+)e_i^{\langle \ga_i^\vee, \xi \rangle+1} + U(\hfn_+)r_i( \mathcal{I}),
  \]
  where $r_i$ denotes the algebra automorphism of $U(\hfg)$ corresponding to the reflection $s_i$.
\end{Lem}

\begin{proof}
  The following proof is essentially same as a part of the proof of \cite[Theorem 3.4]{MR826100}.

  It follows from the $\ad(e_i)$-invariance of $\mathcal{I}$ that 
  \[  U(\hfn_+)\mathcal{I} = \C[e_i]\mathcal{I} \subseteq \mathcal{I} + U(\hfn_+)e_i.
  \]
  Hence (ii) implies $\mathrm{Ann}_{U(\hfn_i)}v = \mathcal{I}$, for $U(\hfn_+) = U(\hfn_+)e_i \oplus U(\hfn_i)$.
  By \cite[Lemma 3.8]{MR1104219}, there is a $\hfg$-module automorphism $r_i'$ on $V$ satisfying 
  $r_i' (v) \in \C^\times v'$ and $\mathrm{Ad}(r_i')x = r_i(x)$ for $x \in \hfg$.
  Hence by applying $r_i$ to the equality $\mathrm{Ann}_{U(\hfn_i)}v = \mathcal{I}$, we have
  $\mathrm{Ann}_{U(\hfn_i)} v'= r_i(\mathcal{I})$.
  Now, since (iii) implies
  \[ \mathrm{Ann}_{U(\hfn_+)}v' = U(\hfn_+)e_i^{\langle \ga_i^{\vee}, \xi\rangle +1} + U(\hfn_+)\mathrm{Ann}_{U(\hfn_i)}v'
  \]
  by \cite[Proposition 3.2]{MR826100}, the assertion is proved.  
\end{proof}

For $1 \le j \le n$ and $0 \le i \le j$, define $\rho_{i,j}\colon \hgDre \to \Z_{\ge 0}$ by
  \[ \rho_{i,j}(\ga) = \sum_{k =j}^n \max \Big\{0, -\big\langle \ga^{\vee}, w_i w_{[j+1,k]}\gL^k\big\rangle\Big\},
  \] 
and put $\hat{\gD}_+^{\mathrm{re}}(i,j) = \{ \ga \in \hat{\gD}_+^{\mathrm{re}} \mid \rho_{i,j}(\ga) > 0\}$.
When $j<n$, we have
\begin{equation}\label{eq:ind}
   \rho_{0,j}(\ga) = \rho_{j+1,j+1}(\ga) + \max\big\{0, -\langle \ga^{\vee},\gL^j\rangle\big\} = \rho_{j+1,j+1}(\ga)
   \ \ \ \text{for} \  \ga \in \hgDre_+,
\end{equation}
which implies $\hgDre_+(0,j) =\hgDre_+(j+1,j+1)$.

\begin{Lem}\label{Lem:values}
%  For $1 \le k \le n$ and $\ga \in \hat{\gD}^{\mathrm{re}}$, set
%  \[ \rho_k(\ga) = \sum_{i=k}^n \max \Big\{ 0, -\langle \ga^\vee, w_{[k,i]}\gL^i \rangle\Big\}.
%  \] 
  Assume $1 \le i \le j \le n$. \\
    {\normalfont(i)} If $\fg = B_n$, then  
    \begin{align*}
        \hat{\gD}_+^{\mathrm{re}}&(i,j) \subseteq \big\{\ga_{p,i-1} \bigm| 1 \le p < i \Big\} \sqcup 
                                 \big\{\ga_{p,q}\bigm| 1 \le p \le j \le q < n, \ p\neq i \big\} \\
                &\sqcup \Big\{- (\ga_{i,n} + \ga_{q,n}) + \gd \Bigm| j <  q \le n\Big\}
                 \sqcup \Big\{-(\ga_{p,n} + \ga_{q,n}) + \gd \Bigm| j < p < q \le n \Big\}.
    \end{align*}
  {\normalfont(ii)} If $\fg = C_n$, then
    \begin{align*}
        \hat{\gD}_+^{\mathrm{re}}(i,j) \subseteq \big\{\ga_{p,i-1} \bigm| 1 \le p < i \big\} &\sqcup 
                                 \big\{\ga_{p,q}\bigm| 1 \le p \le j \le q < n, \ p\neq i \big\} \\
                &\sqcup   \Big\{- (\ga_{i,n} + \ga_{q,n-1}) + \gd  \Bigm| q =i \ \text{or} \ j < q \le n\Big\}\\
                &\sqcup \Big\{-(\ga_{p,n} + \ga_{q,n-1}) + \gd  \Bigm| j < p \le q \le n \Big\}.
    \end{align*}
\end{Lem}

\begin{proof}
  We prove the assertion (i) only (the proof of (ii) is similar).
  Note that the following two containments hold:
  \begin{align}
     \hat{\gD}_+^{\mathrm{re}}(1,j) &\subseteq \tau\Big(\hat{\gD}_+^{\mathrm{re}}(j+1,j+1)\Big) \ \ \text{for} \ j<n
     , \ \text{and} \label{eq:containment2} \\
     \hat{\gD}_+^{\mathrm{re}}(i+1,j) &\subseteq s_i\Big(\hat{\gD}_+^{\mathrm{re}}(i,j)\Big) \sqcup \{\ga_i\} \ \  \text{for} \
     0 < i < j. \label{eq:containment1}
  \end{align}
  In fact, (\ref{eq:containment2}) holds since we have
  \begin{equation*}\label{eq:equal}
    \rho_{1,j} (\ga) = \rho_{0,j}(\tau \ga) = \rho_{j+1,j+1}(\tau \ga)
    \ \ \ \text{for} \ \ga \in \hat{\gD}_+^{\mathrm{re}}
  \end{equation*}
  by (\ref{eq:ind}), and (\ref{eq:containment1}) holds since $\rho_{i+1,j}(\ga) = \rho_{i,j}(s_i \ga)$.
  Then the assertion can be proved inductively from $\hat{\gD}_+^{\mathrm{re}}(1,n) = \emptyset$ using these containments.
\end{proof}

Now let us prove Proposition \ref{Prop:annihilator}.
First assume that $\fg$ is of type $B$.
We verify the assertion
\begin{flushleft}
$
\displaystyle{(\mathrm{B}_{i,j}) \ \ \ \ \ \ \ \ \ \ \ \ \ \ \mathrm{Ann}_{U(\hfn_+)}v(i,j) 
= U(\hfn_+)\bigg(\sum_{\ga \in \hat{\gD}_+^{\mathrm{re}}} \C x_{\ga}^{\rho_{i,j}(\ga)+1} + t\fh[t]\bigg)}
$
\end{flushleft}
by the induction on $(i,j)$, which with $i = j =1$ is just the proposition.
$(\mathrm{B}_{0,n})$ is obvious since $D(0,n) = D(\gL^n)$ is a trivial $\hfn_+$-module and $\rho_{0,n}(\ga) = 0$
for all $\ga \in \hat{\gD}_+^{\mathrm{re}}$.
We easily see that ($\mathrm{B}_{j+1,j+1}$) implies ($\mathrm{B}_{0,j}$)
from (\ref{eq:ind}) and $v(0,j) = v_{\gL^j} \otimes v(j+1,j+1)$,
and it is also easy to check that ($\mathrm{B}_{0,j}$) implies ($\mathrm{B}_{1,j}$)
since $D(1,j) = F_\tau D(0,j)$ by Lemma \ref{Lem:tau} and $\rho_{1,j}(\ga) = \rho_{0,j}(\tau\ga)$.
It remains to show that ($\mathrm{B}_{i,j}$) with $0 <i < j$ implies ($\mathrm{B}_{i+1,j}$).
Let $\xi(i,j) = \sum_{k=j}^n w_iw_{[j+1,k]}\gL^k$, which is the weight of $v(i,j)$.
Since $\langle \ga_i^\vee, w_i w_{[j+1,k]}\gL^k \rangle \ge 0$ for all $k\ge j$ by Lemma \ref{Lem:BC} (ii), we have
\[ f_i^{\langle \ga_i^{\vee}, \xi(i,j)\rangle}v(i,j) \in \C^\times v(i+1,j)
\]
and $\rho_{i+1,j}(\ga_i) = \rho_{i,j}(-\ga_i)=\langle \ga_i^\vee, \xi(i,j)\rangle$.
Therefore it suffices to show the $\ad(e_i)$-invariance of the left $U(\hfn_i)$-ideal
\[ \mathcal{I}_{i,j} = U(\hfn_i)\bigg(\sum_{\ga \in \hat{\gD}_+^{\mathrm{re}} 
   \setminus \{ \ga_i \}} \C x_\ga^{\rho_{i,j}(\ga)+1} + t\fh[t]\bigg)
\]
by Lemma \ref{Lem:proceed}.
Note that, 
if $\gb \in \hat{\gD}_+^{\mathrm{re}}$ is in the form $\gb = l \ga + \ga_i$ with some $\ga \in \hat{\gD}_+$ 
and $l \in \Z_{>0}$, then $\rho_{i,j}(\gb) = 0$ holds.
In fact, the condition implies $\gb \in \hgDre_+ + \gd$, or
\begin{align*}
   \gb \in \big\{\ga_{p,i} \bigm| p < i\big\} &\!\sqcup\! \big\{\ga_{i,q} \bigm| q > i\big\} \!
                                  \sqcup \!\big\{\ga_{i,n} + \ga_{q,n} \bigm| q \neq i \big\}
           \sqcup\!\big\{- \ga_{p,i-1} + \gd \bigm| p<i\big\}\\ &\!\sqcup\! \big\{- \ga_{i+1,q} + \gd \bigm|  q > i\big\} 
   \!\sqcup\! \Big\{-\big(\ga_{i+1,n} + \ga_{q,n}\big) + \gd \Bigm| q \neq i,i+1\Big\},
\end{align*}
and hence $\rho_{i,j}(\gb) = 0$ follows from Lemma \ref{Lem:values} and $0<i<j$.
From this, the $\ad(e_i)$-invariance of $\mathcal{I}_{i,j}$ is immediately proved.

Next assume that $\fg = C_n$, and define a subset $S_j$ of $\hat{\gD}_+^{\mathrm{re}} \times \hat{\gD}_+^{\mathrm{re}}$
for $1 \le j \le n$ by
\[ S_j= \bigg\{ \Big(-\big(\ga_{p,n} + \ga_{q,n-1}\big) + \gd, 
        -\big(\ga_{q,n} + \ga_{q,n-1}\big) + \gd \Big)\biggm| j \le p < q \le n\bigg\}.
\]
We verify the assertion $(\mathrm{C}_{i,j})$: $\mathrm{Ann}_{U(\hfn_+)}v(i,j) = \mathcal{J}_{i,j}$ by the induction on $(i,j)$,
where $\mathcal{J}_{i,j}$ is a $U(\hfn_+)$-ideal defined by 
\[ \mathcal{J}_{i,j} = U(\hfn_+)\bigg(\sum_{\ga \in \hat{\gD}_+^{\mathrm{re}}} \C x_{\ga}^{\rho_{i,j}(\ga)+1} +
    \sum_{\begin{smallmatrix} (\ga,\gb) \in s_is_{i+1}\cdots s_{j-1}(S_j) \\ 1 \le k < \rho_{i,j}(\ga)/2+1 \end{smallmatrix}} 
    \C x_\ga^{\rho_{i,j}(\ga) -2k +1} x_\gb^{k}+ t\fh[t]\bigg).
\]
Here we set $w(S_j) = \big\{(w\ga, w\gb) \bigm| (\ga,\gb) \in S_j\big\}$ for $w \in \hW$. 
Note that ($\mathrm{C}_{1,1}$) is just the proposition, and ($\mathrm{C}_{0,n}$) is obvious.

Let us show that ($\mathrm{C}_{j+1,j+1}$) implies ($\mathrm{C}_{0,j}$).
Since $v(0,j) = v_{\gL^j} \otimes v(j+1,j+1)$, we have $\mathrm{Ann}_{U(\hfn_+)} v(0,j) = \mathrm{Ann}_{U(\hfn_+)} v(j+1,j+1)$,
and hence it suffices to show that $\mathcal{J}_{0,j} = \mathcal{J}_{j+1,j+1}$.
We have $\rho_{0,j}(\ga) = \rho_{j+1,j+1}(\ga)$ for $\ga \in \hgDre_+$ by (\ref{eq:ind}),
and a direct calculation shows
\begin{equation}\label{eq:S}
  s_0s_1 \cdots s_{j-1}S_{j} 
  = S_{j+1} \sqcup \bigg\{ \Big(\ga_{1, q-1}, -(\ga_{q,n} + \ga_{q,n-1}) + \gd \Big)\biggm| j < q \le n\bigg\}.
\end{equation}
Hence it is enough to check that 
\begin{equation}\label{eq:ad}
  x_{\ga}^{\rho_{j+1,j+1}(\ga) -2k +1} x_{\gb}^{k} \in \mathcal{J}_{j+1,j+1}
\end{equation}
for $\ga = \ga_{1,q-1}$, $\gb = -(\ga_{q,n} + \ga_{q,n-1}) + \gd$ with $j < q$ and $1 \le k < \rho_{j+1,j+1}(\ga)/2+1$.
Since $\gb = \ga_0 + 2\ga$, it is directly checked for $l_1 \in \Z_{\ge 2}$ and $l_2 \in \Z_{\ge 0}$ that 
\begin{equation}\label{eq:ad2}
  e_0x_\ga^{l_1}x_\gb^{l_2} = x_\ga^{l_1}x_\gb^{l_2}e_0 + a_1x_{\ga}^{l_1-1}x_\gb^{l_2}x_{\ga_0 + \ga}
                                            +a_2 x_\ga^{l_1-2} x_\gb^{l_2+1}
\end{equation}
with $a_1,a_2 \in \C^\times$. 
Since $e_0$ and $x_{\ga_0 + \ga}$ belong to $\mathcal{J}_{j+1,j+1}$,
(\ref{eq:ad}) is proved inductively using this equality from $x_{\ga}^{\rho_{j+1,j+1}(\ga)+1} \in \mathcal{J}_{j+1,j+1}$.

Finally, let us show that ($\mathrm{C}_{i,j}$) with $0 \le i < j$ implies ($\mathrm{C}_{i+1,j}$). 
It is easy to see that $f_i^{\langle \ga_i^\vee, \xi(i,j)\rangle} v(i,j) \in \C^\times v(i+1,j)$ and
$\rho_{i+1,j}(\ga_i^\vee) = \langle \ga_i^\vee, \xi(i,j) \rangle$, where $\xi(i,j)$ is the weight of $v(i,j)$.
Hence by Lemma \ref{Lem:proceed}, it suffices to show the $\ad(e_i)$-invariance of the left $U(\hfn_i)$-ideal
\[ \mathcal{I}_{i,j} = U(\hfn_i)\bigg(\sum_{\ga \in \hat{\gD}_+^{\mathrm{re}}\setminus\{\ga_i\}} \C x_{\ga}^{\rho_{i,j}(\ga)+1} +
    \sum_{\begin{smallmatrix} (\ga,\gb) \in s_is_{i+1}\cdots s_{j-1}(S_j) \\ 1 \le k < \rho_{i,j}(\ga)/2+1 \end{smallmatrix}} 
    \C x_\ga^{\rho_{i,j}(\ga) -2k +1} x_\gb^{k}+ t\fh[t]\bigg).
\]
First, assume that $0 < i <j$. 
It is checked in a similar way as above that, 
if $\gb \in \hat{\gD}_+^{\mathrm{re}}$ is in the form $\gb = l \ga + \ga_i$ with $\ga \in \hat{\gD}_+$ and 
$l \in \Z_{>0}$, then $\rho_{i,j}(\gb) = 0$.
Since $[x_\ga,x_\gb] = 0$ for $(\ga,\gb) \in s_i \cdots s_{j-1} (S_j)$,
the $\ad(e_i)$-invariance of $\mathcal{I}_{i,j}$ is proved from this.
%, together with $[x_\ga, x_\gb] = 0$ for $(\ga,\gb) \in s_i s_{i+1} \cdots s_{j-1}S_j$. 
Next, assume that $i = 0$.
Using $\hgDre_+(0,j) =\hgDre_+(j+1,j+1)$,
it is similarly checked that, if $\gb \in \hat{\gD}_+^{\mathrm{re}}$ 
is in the form $\gb = l\ga + \ga_0$ with $\ga \in \hat{\gD}_+$ and $l \in \Z_{>0}$,
then $\rho_{0,j}(\gb) = 0$, or 
\[ l = 2, \ \ \ga = \ga_{1,q-1} \ \ \text{and} \ \ \gb = -(\ga_{q,n} + \ga_{q,n-1}) + \gd
\]
for some $j < q \le n$.
In the latter case, we see from (\ref{eq:ad2}) that 
\[ U(\hfn_0)\bigg( \C x_{\ga_0+\ga} + \sum_{0 \le k < \rho_{0,j}(\ga)/2+1} \C x_{\ga}^{\rho_{0,j}(\ga)-2k +1}x_\gb^k\bigg)
\]
is $\ad(e_0)$-invariant.
Now the $\ad(e_0)$-invariance of $\mathcal{I}_{0,j}$ is easily proved, using (\ref{eq:S}).
The proof is complete.

\subsection{$q$-characters}\label{Subsection:q-char}

Here we recall the definition of $q$-characters and some results on them,
which are necessary in Subsections \ref{Proof2} and \ref{Proof3}.
For a finite-dimensional $\ell$-weight module $V$, define its $\ell$-weight set $\wt_\ell V$
and \textit{$q$-character} $\chq V$ by
\[ \wt_\ell V = \{ \bm{\rho} \in P_q \mid V_{\bm{\rho}} \neq 0\} \ \ \text{and} \ \  
   \chq V= \sum_{\bm{\rho} \in P_q} (\dim V_{\bm{\rho}}) \bm{\rho} \in \Z [P_q]
\]
respectively.
For finite-dimensional $\ell$-weight modules $V_1, V_2$ and $\bm{\rho} \in P_q$, 
it follows that
\[ (V_1 \otimes V_2)_{\bm{\rho}} = \bigoplus_{\bm{\nu} \in P_q} (V_1)_{\bm{\nu}} \otimes (V_2)_{\bm{\rho}\bm{\nu}^{-1}}, 
\]
and therefore $\chq V_1 \otimes V_2 = \chq V_1 \cdot \chq V_2$ holds \cite{MR1745260}.
For $i \in I$ and $a \in \C(q)^\times$, define $\bm{\ga}_{i,a} \in P_q$ by
\[ \bm{\ga}_{i,a} = \bm{\pi}_{2, a}^{(i)} \prod_{j \neq i} \Big(\bm{\pi}_{-c_{j,i},a}^{(j)}\Big)^{-1}.
\]
Let $Q_q^+$ denote the submonoid generated by $\{\bm{\ga}_{i,a}\mid i \in I, a \in \C(q)^\times\}$,
and $Q_q$ the corresponding subgroup. %We call $Q_q$ the \textit{$\ell$-root lattice}.
We write $\bm{\rho} \le \bm{\nu}$ for $\bm{\rho}, \bm{\nu} \in P_q$ if
$\bm{\nu} \bm{\rho}^{-1} \in Q_q^+$ holds.

\begin{Prop}[{\cite[Theorem 4.1]{MR1810773}}]\label{Prop:weight}
  For every $\bm{\rho} \in P^+_q$, $\bm{\nu} \in \wt_\ell L_q(\bm{\rho})$ implies
  $\bm{\nu} \le \bm{\rho}$.
\end{Prop}

The following proposition is proved from the study of $U_q(\bL \mathfrak{sl}_2)$-modules in \cite{MR1137064,MR1357195,MR1745260}.

\begin{Prop}\label{Prop:sl2}
  Assume that $\fg= \mathfrak{sl}_2$. Then the following statements hold, where we omit the index $i$. \\
  {\normalfont(i)} 
  \[ \chq L_q(\bm{\pi}_{m,a}) =\bm{\pi}_{m,a}\sum_{0 \le k \le m}\prod_{0\le j \le k-1} \bm{\ga}_{aq^{m-2j}}^{-1}.
  \]
  {\normalfont(ii)}
  If $V$ is an $\ell$-highest weight module with $\ell$-highest weight $\bm{\pi}_{m,a}$,
  then we have
  \[ \chq L_q(\bm{\pi}_{m,a}) 
     \le \chq V \le \prod_{1 \le j \le m} \chq L_q(\bm{\varpi}_{aq^{m -2j + 1}}) 
     = \bm{\pi}_{m,a}\prod_{1 \le j \le m}(1+\bm{\ga}_{aq^{m-2j+2}}^{-1}),
  \] 
  where the inequality $f \le g$ means $g -f \in \Z_{\ge 0}[P_q]$.
  In particular, the dimension of each $\ell$-weight space of $V$ is at most $1$. 
\end{Prop}

Recall the map $P_q \ni \bm{\rho} \mapsto \bm{\rho}_J \in P_{q,J}$ for a subset $J\subseteq I$ defined 
in Subsection \ref{subsection:minimal}.
The following proposition is an easy consequence of results in \cite[Section 3]{MR1810773}.

\begin{Prop}\label{Prop:reduction}
  Let $V$ be a finite-dimensional $\ell$-weight module, and $J \subseteq I$ a subset such that $\fg_J$ is simple.
  For an $\ell$-weight vector $v \in V_{\bm{\rho}}$, let $W = U_q(\bL\fg_J)v$.
  Assume that a vector $w \in W$ is $\ell$-weight with respect to $U_q(\bL\fh_J)$, and its $\ell$-weight is
  \[ \bm{\rho}_J \prod_{i \in J, a \in \C(q)^\times} (\bm{\ga}_{i,a})_J^{v(i,a)} \in P_{q,J}
  \]
  with some integers $v(i,a)$. 
  Then $w$ is also $\ell$-weight with respect to $U_q(\bL\fh)$,
  and its $\ell$-weight is $\bm{\rho} \prod \bm{\ga}_{i,a}^{v(i,a)} \in P_q$.
\end{Prop}

%\begin{Prop}[{\cite[Lemma 5.4]{MR2468483}}]
%  Let $V$ be a finite-dimensional $\ell$-weight module, $j \in I$ and $0 \neq v \in V_{\bm{\nu}}$ for some $\bm{\nu} \in P_q$.
%  Set $W = U_q(\bL\fg_j)v \subset V$.
%  If $W \cap V_{\bm{\nu}} \neq 0$, then $\bm{\nu} \bm{\rho}^{-1}\in Q_{q,i}$. 
%\end{Prop}

Let $j \in I$.
We say $\bm{\rho} = \prod_{i \in I, a \in \C(q)^\times} \bm{\varpi}_{i,a}^{u(i,a)} \in P_q$ is \textit{$j$-dominant}
if $u(j,a) \ge 0$ for all $a \in \C(q)^\times$.
The following proposition was established by Hernandez.

\begin{Prop}[{\cite[Lemma 5.6]{MR2468483}}]\label{Prop:3cond}
  Let $\bm{\rho} \in P_q^+$, and $\bm{\nu} \in \wt_{\ell}L_q(\bm{\rho}) \setminus \{\bm{\rho}\}$.
  Then there exist some $j \in I$ and $\bm{\nu}' \in \wt_\ell L_q(\bm{\rho})$ such that 
  \begin{enumerate}
    \item[\normalfont (i)] $\bm{\nu}'$ is $j$-dominant, 
    \item[\normalfont (ii)] $\bm{\nu}' \in \bm{\nu}\prod_{a \in \C(q)^\times} \bm{\ga}_{j,a}^{\Z_{\ge 0}}$
                                     and $\bm{\nu}' > \bm{\nu}$,
    \item[\normalfont (iii)] $\Big(U_q(\bL\fg_j)L_q(\bm{\rho})_{\bm{\nu}}\Big) \cap L_q(\bm{\rho})_{\bm{\nu}'} \neq 0$.  
  \end{enumerate} 
\end{Prop}

\begin{Def}[\cite{MR1810773}]\normalfont
  Assume that $\bm{\rho} \in P_q$ is in the form 
  $\prod_{i\in I,k \in \Z} \bm{\varpi}_{i,aq^k}^{u(i,k)}$ with some $a \in \C(q)^\times$ and $u(i,k) \in \Z$. 
  We say $\bm{\rho}$ is \textit{right-negative}
  if $k_{\max} = \max\{ k \in \Z \mid u(i,k) \neq 0 \ \text{for some} \ i \in I \}$ satisfies 
  $u(i,k_{\max}) \le 0$ for all $i \in I$.
\end{Def}

Note that if $\bm{\rho}$ is right-negative, $\bm{\rho}$ is not dominant.
We easily see that $\bm{\ga}_{i,a}^{-1}$ are right-negative.

\begin{Lem}[\cite{MR1810773}]\label{Lem:order}
  {\normalfont(i)} If $\bm{\rho}$ is right-negative and $\bm{\nu} \le \bm{\rho}$, then $\bm{\nu}$ is also right-negative.\\
  {\normalfont(ii)} If $\bm{\rho} \in \wt_\ell L_q(\bm{\varpi}_{i,a}) \setminus \{\bm{\varpi}_{i,a}\}$,
  then $\bm{\rho} \le \bm{\varpi}_{i,a} \bm{\ga}_{i,aq_i}^{-1}$.
  In particular $\bm{\rho}$ is right-negative by {\normalfont(i)}.
\end{Lem}

%For $i \in I$ and $a \in \C(q)^\times$, $L_q(\bm{\varpi}_{i,a})$ is called the \textit{fundamental representation} 
%associated with $i$ and $a$.

\subsection{Proof of \boldmath$M(\gl) \twoheadrightarrow L(\bm{\pi})$}\label{Proof2}

For $1 \le i \le j \le n$ and $p \in \Z$, define $v_p(i,j) \in L_q(\bm{\pi})$ by
\[ v_p(i,j) = \begin{cases} x_{i,p}^- x_{i+1,0}^-x_{i+2,0}^-\cdots x_{j,0}^-v_{\bm{\pi}} 
                             & \text{if} \ \bm{\pi} \ \text{satisfies {\normalfont(I)}},\\
                            x_{j,p}^- x_{j-1,0}^-x_{j-2,0}^- \cdots x_{i,0}^-v_{\bm{\pi}}
                             & \text{if} \ \bm{\pi} \ \text{satisfies {\normalfont(II)}},
               \end{cases}
\]
where $v_{\bm{\pi}}$ is an $\ell$-highest weight vector and (I), (II) are the conditions in Theorem \ref{Thm:Classification}.
Let $(a_i)_{i \in I}$ be the sequence of rational functions in Theorem \ref{Thm:Classification} associated with $\bm{\pi}$.
The following proposition is crucial in this subsection.

\begin{Prop}\label{Lem:essential}
  Let $1 \le i \le j \le n$.\\
  {\normalfont (i)} For all $p \in \Z$, $v_p(i,j)$ is a scalar multiple of $v_0(i,j)$. \\
  {\normalfont (ii)} The vector $v_0(i,j)$ is a simultaneous $U_q(\bL\fh)$-eigenvector, and its $\ell$-weight is 
                     $\bm{\pi}\prod_{i \le k \le j} \bm{\ga}_{k,a_kq_k^{\gl_k}}^{-1}$.
\end{Prop}

Let us assume this proposition for a moment.
%Take $1 \le i \le j \le n$ arbitrarily.
To prove $M(\gl) \twoheadrightarrow L(\bm{\pi})$, 
we need to check that the vector $\ol{v}_{\bm{\pi}} = 1 \otimes v_{\bm{\pi}} \in L(\bm{\pi})$ 
satisfies the defining relations of $M(\gl)$.
Using the commutativity $[x_{k,r}^-, x_{l,s}^-] = 0$ for $|k -l| \ge 2$, 
we easily see that Proposition \ref{Lem:essential} (i) implies
\begin{align*}
  &[x_{i,1}^- ,[x_{i+1,0}^-,\ldots, [x_{j-1,0}^-, x_{j,0}^-]\!\ldots ]]v_{\bm{\pi}} \in V_q(\gl) \subseteq L_q(\bm{\pi})
    \ \ \ \text{if} \ \bm{\pi} \  \text{satisfies (I)}, 
    \ \ \ \text{and}\\
  &[x_{j,1}^- ,[x_{j-1,0}^-,\ldots, [x_{i+1,0}^-, x_{i,0}^-]\!\ldots ]]v_{\bm{\pi}} \in V_q(\gl) \subseteq L_q(\bm{\pi})
    \ \ \ \text{if} \ \bm{\pi} \ \text{satisfies (II)} 
\end{align*}
for all $1 \le i \le j \le n$.
Here $V_q(\gl)$ denotes (by abuse of notation) 
the $U_q(\fg)$-submodule of $L_q(\bm{\pi})$ generated by $v_{\bm{\pi}}$.
By the definition of $L(\bm{\pi})$, this implies 
\[ (f_{\ga} \otimes t)\ol{v}_{\bm{\pi}}
   \in V(\gl) \subseteq L(\bm{\pi}) \ \ \ \text{if} \ \ga = \ga_i + \cdots + \ga_j \in \gD_+^1.
\]
%where $\ga = \ga_i + \cdots + \ga_j \in \gD_+^1$.
Since the restriction of the surjection $L(\bm{\pi}) \twoheadrightarrow V(\gl,0)$ in Lemma \ref{Lem:elementary_minimal} (i)
on $V(\gl) \subseteq L(\bm{\pi})$ is an isomorphism, 
this implies $(f_\ga \otimes t)\ol{v}_{\bm{\pi}} = 0$ for all $\ga \in \gD_+^1$.
The other relations are proved from this relation or follow from Lemma \ref{Lem:elementary_minimal}.
Hence the assertion is proved.

The following lemma was shown in \cite[Lemma 3.6]{MR1347873}.% for type $B$, and the proofs for the other types are the same.

\begin{Lem}\label{Lem:two-case}
  Assume that $\fg$ is of type $ABC$. Let $i \in I$, $\mu \in P^+$ be such that 
  $\langle \ga_j^\vee, \mu \rangle = 0$ for $i < j <n$, 
  and $\bm{\rho} \in P^+_q$ such that $L_q(\bm{\rho})$ is a minimal affinization of $V_q(\mu)$.
  Then for all $p \in \Z$, the vector $v_p(i,n)$ is a scalar multiple of $v_0(i,n)$,
  where $v_p(i,n) \in L_q(\bm{\rho})$ are defined similarly as above.
\end{Lem}

We verify Proposition \ref{Lem:essential} by the induction on $j -i$, assuming that $\bm{\pi}$ satisfies condition (I).
(The proof for condition (II) is similar).
In view of Corollary \ref{Cor:restriction}, the case $i = j$ follows from Theorem \ref{Thm;typeA} and 
Proposition \ref{Prop:sl2}.
Let $i \le j -1$.
Since $\gl_j=0$ implies $v_p(i,j) = 0$ for all $p \in \Z$, we may assume $\gl_j > 0$.
By the induction hypothesis, $v_0(i+1,j)$ is a simultaneous $U_q(\bL\fh)$-eigenvector with $\ell$-weight  
$\bm{\nu} = \bm{\pi}\prod_{i+1 \le k \le j} \bm{\ga}_{k,a_k q_k^{\gl_k}}^{-1}$.
Let $W = U_q(\bL\fg_i)v_0(i+1,j) \subseteq L_q(\bm{\pi})$.
Then $W$ is an $\ell$-highest weight $U_q(\bL\fg_i)$-module,
and its $\ell$-highest weight with respect to $U_q(\bL \fh_i)$ is 
\[ \bm{\nu}_i(u)%\Big(\bm{\pi}\prod_{i+1 \le l \le j} \bm{\ga}_{l,a_l q_l^{\gl_l}}^{-1}\Big)_i
   = \bm{\pi}_i(u) \cdot \Big(\bm{\ga}_{i+1,a_{i+1}q^{\gl_{i+1}}}^{-1}\Big)_i(u)
   = \prod_{1 \le k \le \gl_i -c_{i,i+1}} (1 -a_i q_i^{\gl_i - 2k +1}u).
   %\bm{\pi}_{-c,{a_{i+1} q^{\gl_{i+1}}_{i+1}}}^{(i)} 
%   = \prod_{1 \le r \le \gl_i - c_{i+1,i}}(1 - a_iq_i^{\gl_i -c_{i+1,i} -2r +1})\bm{\pi}^{(i)}_{\gl_i - c, a_iq_i^{c}},
\]
Hence by Proposition \ref{Prop:sl2} (ii), each $\ell$-weight space of $W$ is $1$-dimensional.
Let us assume that the assertion (i) of Proposition \ref{Lem:essential} does not hold,
which implies that the dimension of the space $\sum_p\C(q) x_{i,p}^-v_0(i+1,j)$ is at least $2$.
Hence we can take $Y_s \in \sum_p \C(q) x_{i,p}^-$ for $s =1,2$ such that 
$0 \neq Y_s v_0(i+1,j) \in L_q(\bm{\pi})_{\bm{\nu} \bm{\ga}_{i,b_s}^{-1}}$ for some $b_s \in a_iq_i^\Z$ with $b_1 \neq b_2$,
using Proposition \ref{Prop:reduction}.
Let $l = \min\{i < l' < j \mid \gl_{l'} > 0 \}$, which exists by Lemma \ref{Lem:two-case} and Corollary
\ref {Cor:restriction}.
We easily see that $\bm{\nu}\bm{\ga}_{i,b_s}^{-1}$ is not $(l+1)$-dominant,
and hence there exists
some $p_s \in \Z$ such that $x_{l+1,p_s}^+ Y_sv_0(i+1,j)\neq 0$.
From this and the induction hypothesis, we see 
that $Y_sx_{i+1,0}^- \cdots x_{l,0}^-v_0(l+2,j)$ is a nonzero $\ell$-weight vector,
and by Proposition \ref{Prop:weight} its $\ell$-weight is
\[ \bm{\nu}\bm{\ga}_{i,b_s}^{-1}\bm{\ga}_{l+1,a_{l+1} q^{\gl_{l+1}}} = 
   \bm{\pi}\bm{\ga}_{i,b_s}^{-1}\prod_{i+1 \le k \le j, k \neq l+1} \bm{\ga}_{k,a_kq^{\gl_k}}^{-1}.
\]
By repeating this argument, we finally see that $Y_sv_0(i+1,l)$ is a nonzero $\ell$-weight vector
with $\ell$-weight $\bm{\pi}\bm{\ga}_{i,b_s}^{-1}\prod_{i+1 \le k \le l} \bm{\ga}_{k,a_{k}q^{\gl_k}}^{-1}$
for $s = 1,2$.
Since $b_1 \neq b_2$, this contradicts Lemma \ref{Lem:two-case}, and the assertion (i) is proved.
Now the assertion (ii) is easily proved from (i) and the induction hypothesis. 
The proof is complete.

\begin{Rem}\normalfont
  In type $B$, $M(\gl) \twoheadrightarrow L(\bm{\pi})$ is also proved in \cite[Proposition 3.22]{MR2587436}
  using the Frenkel-Mukhin algorithm.
\end{Rem}

\subsection{Proof of \boldmath$L(\bm{\pi}) \twoheadrightarrow D(w_\circ \xi_1,\ldots,w_\circ \xi_n)$}\label{Proof3}

In this proof we use the following obvious fact repeatedly without further mention:
if $M,N$ are $\fg[t]$-modules and $N$ is cyclic with generator $v$, 
then a homomorphism $M \to N$ containing $v$ in its image is surjective.
We begin with the proof of the following lemma (note that $\bar{i} = i$ holds for all $i \in I$ in type $BC$).

\begin{Lem}\label{Lem:KR}
  For every $i \in I$, $m \in \Z_{\ge 0}$ and $a \in \C(q)^\times$,
  \[ L\big(\bm{\pi}_{m,a}^{(i)}\big) \cong D\big(-m\varpi_i + \big\lceil d_i m/ 2 \big\rceil \gL_0\big).
  \]
\end{Lem}

\begin{proof}
  If $d_im/2 \in \Z$, the assertion follows from \cite[Proposition 5.1.3]{MR2238884} (see also \cite[Theorem 4]{MR2323538}).
  Hence we may assume $d_im/2 \notin \Z$, which is equivalent to that $\ga_i$ is short and
  $m = 2 k +1$ for some $k \in \Z_{\ge 0}$.
  Then there is an injective homomorphism
  \[ L(\bm{\pi}_{m,a}^{(i)}) \hookrightarrow D(-2\varpi_i + \gL_0)^{\otimes k} \otimes D(-\varpi_i + \gL_0)
  \]
  by \cite[Theorem 2.2]{MR2238884}, which implies
  \[ L(\bm{\pi}_{m,a}^{(i)}) \cong D(\underbrace{-2\varpi_i + \gL_0, \ldots,-2\varpi_i + \gL_0}_k, -\varpi_i + \gL_0).
  \]
  Since $w_\circ w_{[1,i]}(\gL_0)= -2\varpi_i + \gL_0$ and $w_\circ w_{[1,i]}(\varpi_i+\gL_0)= -\varpi_i + \gL_0$
  hold by Lemma \ref{Lem:BC} (i), we have
  \begin{align*}
    D(\underbrace{-2\varpi_i \!+ \!\gL_0,\ldots,-2\varpi_i\! +\! \gL_0}_k, -\varpi_i\! +\! \gL_0) 
    &= F_{w_\circ w_{[1,i]}} \Big(D(\gL_0)^{\otimes k} \!
    \otimes D(\varpi_i + \gL_0)\Big) \\ &\cong D\big(\!-\!m\varpi_i + (k+1)\gL_0\big)
  \end{align*}
  by Proposition \ref{Prop:character2}. The assertion is proved.
\end{proof}

Hence in type $B$, $L(\bm{\pi}) \twoheadrightarrow D(w_\circ \xi_1,\ldots,w_\circ \xi_n)$ 
follows from \cite[Proposition 3.21]{MR2587436}. 
(Note that this proposition does not imply our assertion in type $C$.) 

In the rest of this subsection, we assume that $\fg$ is of type $C$.
%Let $\varpi_{i,j} = \varpi_i + \varpi_j$ for $1 \le i \le j \le n-1$.
For the proof of the assertion in this type, we need the following lemma.

\begin{Lem}\label{Lem:fundamentals}
  Let $1 \le r < s \le n-1$, and assume that $\bm{\rho}$ is an element of $P_{\bm{A}}^+$ such that $L_q(\bm{\rho})$ 
  is a minimal affinization of $V_q(\varpi_r + \varpi_s)$. 
  Then we have
  \[ L(\bm{\rho}) \cong D(-\varpi_r - \varpi_s + \gL_0).
  \]  
\end{Lem}

Assuming this lemma for a while, we shall prove $L(\bm{\pi}) \twoheadrightarrow D(w_\circ \xi_1, \ldots, w_\circ \xi_n)$
for $\bm{\pi}$ satisfying condition (I). (The proof for condition (II) is similar.)
The proof is carried out in a similar line as that of \cite[Proposition 3.21]{MR2587436}.
First we recall the following theorem, which is obtained by taking the dual of \cite[Theorem 5.1]{MR1883181} and
using Lemma \ref{Lem:dual}.

\begin{Thm}\label{Thm:Chari}
  Let $i_1, \ldots,i_p \in I$, $b_1,\ldots,b_p \in \C(q)^\times$, and $l_1, \ldots, l_p \in \Z_{>0}$, and assume that  
  \begin{equation}\label{eq:condition}
    b_rq_{i_r}^{-l_r} \notin q^{\Z_{> 0}} b_sq^{-l_s}_{i_s} \ \ \ \text{for all} \ r < s.
  \end{equation}
  Then the submodule of $L_q\big(\bm{\pi}_{l_1,b_1}^{(i_1)}\big)
  \otimes \cdots \otimes L_q\big(\bm{\pi}_{l_p,b_p}^{(i_p)}\big)$
  generated by the tensor product of $\ell$-highest weight vectors is isomorphic to 
  $L_q\Big(\prod_{k =1}^p \bm{\pi}_{l_k,b_k}^{(i_k)}\Big)$.
\end{Thm}

The following corollary is an easy consequence of this theorem.

\begin{Cor}\label{Cor:Chari}
  Assume that $i_1, \ldots,i_p \in I$, $b_1, \ldots, b_p \in \C(q)^\times$, and $l_1, \ldots,l_p \in \Z_{> 0}$ satisfy
  {\normalfont(\ref{eq:condition})}. 
  Then for any sequence $0 =k_0 < k_1 < \cdots <k_{r-1} < k_r = p$, 
  the submodule of $L_q\Big(\prod_{k=1}^{k_1} \bm{\pi}_{l_k,b_k}^{(i_k)}\Big) \otimes 
  \cdots \otimes L_q\Big(\prod_{k=k_{r-1}+1}^{p}\bm{\pi}_{l_k,b_k}^{(i_k)}\Big)$
  generated by the tensor product of $\ell$-highest weight vectors is isomorphic to 
  $L_q\Big(\prod_{k =1}^p \bm{\pi}_{l_k,b_k}^{(i_k)}\Big)$.
\end{Cor}

Let $(a_i)_{i \in I}$ be the sequence in Theorem \ref{Thm:Classification} associated with $\bm{\pi}$.
For each $i \in I$, define $\bm{\pi}^{(i)} \in P_q^+$ by 
\[ \bm{\pi}^{(i)} = \begin{cases}  \bm{\pi}_{\gl_i -p_i, a_iq^{p_i}}^{(i)} & \text{if} \ i^\flat= 0, \\
                                   \bm{\pi}_{p_{i^{\flat}}, a_{i^\flat}q^{-\gl_{i^\flat} + 1}}^{(i^\flat)}
                                   \bm{\pi}_{\gl_i - p_i, a_iq^{p_i}}^{(i)}& \text{otherwise}.
                \end{cases}
\]
By Corollary \ref{Cor:Chari}, 
we have a nonzero $U_q(\bL\fg)$-module homomorphism
\[ L_q(\bm{\pi}) \to L_q\big(\bm{\pi}^{(n)}\big) \otimes \cdots \otimes L_q\big(\bm{\pi}^{(2)}\big) \otimes 
   L_q\big(\bm{\pi}^{(1)}\big),
\]
which induces a $U_{\mathbf{A}}(\bL\fg)$-module homomorphism 
\[ L_{\mathbf{A}}(\bm{\pi}) \to L_{\mathbf{A}}\big(\bm{\pi}^{(n)}\big) 
   \otimes
   \cdots \otimes L_{\mathbf{A}}\big(\bm{\pi}^{(2)}\big) \otimes L_{\mathbf{A}}\big(\bm{\pi}^{(1)}\big).
\]
By applying $\C \otimes_{\mathbf{A}}$ and taking the pull-back, we have a $\fg[t]$-module homomorphism
$L\big(\bm{\pi}\big) \to \bigotimes_{i=n}^{1} L\big(\bm{\pi}^{(i)}\big)$
mapping $\ol{v}_{\bm{\pi}}$ to $\ol{v}_{\bm{\pi}^{(n)}} \otimes \cdots \otimes \ol{v}_{\bm{\pi}^{(1)}}$. 
%Thus, to prove $L(\bm{\pi}) \twoheadrightarrow 
Hence it suffices to show for each $1 \le i \le n$ the existence of a surjective homomorphism 
$L\big(\bm{\pi}^{(i)}\big) \twoheadrightarrow D(w_\circ \xi_i)$, since this induces
\[ L(\bm{\pi}) \twoheadrightarrow U(\fg[t])(\ol{v}_{\bm{\pi}^{(n)}} \otimes \cdots \otimes \ol{v}_{\bm{\pi}^{(1)}})
   \twoheadrightarrow D(w_\circ \xi_n, \ldots, w_\circ \xi_1),
\]
and the last term is isomorphic to $D(w_\circ \xi_1,\ldots,w_\circ \xi_n)$ by definition.
If $p_{i^{\flat}} = 0$ or $i^{\flat} = 0$, this assertion follows from Lemma \ref{Lem:KR}.
Assume that $p_{i^\flat} = 1$ and $i^\flat \neq 0$, and put
\[ \bm{\pi}_1= \bm{\pi}^{(i)}_{\gl_i -p_i -1, a_iq^{p_i-1}} , \ \ \ 
   \bm{\pi}_2= \bm{\varpi}_{i^{\flat},a_{i^\flat}q^{-\gl_{i^\flat}+1}}\bm{\varpi}_{i,a_iq^{\gl_i-1}}.
\]
There is a nonzero homomorphism   
$L_q\big(\bm{\pi}^{(i)}\big) \to L_q(\bm{\pi}_1) \otimes L_q(\bm{\pi}_2 )$ by Corollary \ref{Cor:Chari},
and then using the same argument as above, we obtain a $\fg[t]$-module homomorphism
$L\big(\bm{\pi}^{(i)}\big) \to L\big(\bm{\pi}_1\big) \otimes L(\bm{\pi}_2)$.
By Lemmas \ref{Lem:KR} and \ref{Lem:fundamentals},
this induces a surjective homomorphism
\[ L\big(\bm{\pi}^{(i)}\big) \twoheadrightarrow 
   D\Big(-(\gl_i-p_i-1) \varpi_i + \frac{1}{2}(\gl_i -p_i -1)\gL_0, -\varpi_{i^\flat}-\varpi_i  + \gL_0\Big).
\]
By Proposition \ref{Prop:character2} and Lemma \ref{Lem:BC}, we see that  the right-hand side is isomorphic to
\[ F_{w_\circ w_{[1,i]}}\bigg(D \Big(\frac{1}{2}(\gl_i-p_i-1) \gL_0\Big)\otimes D\big(\varpi_{i - i^\flat} + \gL_0\big) 
    \bigg) \cong D(w_\circ \xi_i),
\]
and hence the assertion is proved.

It remains to show Lemma \ref{Lem:fundamentals}. Fix $1 \le r < s \le n-1$.

\begin{Lem}\label{Lem:decomp}
  As $\fg$-modules,
  \[ D(-\varpi_r -\varpi_s + \gL_0) \cong \bigoplus_{k=0}^{r} V(\varpi_{r-k} + \varpi_{s-k}).
  \] 
\end{Lem}

\begin{proof}
  Let $\hat{J} = \{0,1, \ldots, s-1\} \subseteq \hI$ and $J = \{1,\ldots,s-1\} \subseteq I$,
  and define $\hfg_{\hat{J}}$ by the Lie subalgebra of $\hfg$ generated by $\{e_i,f_i \mid i\in \hat{J}\,\}$ and $\hfh$.
  We also define $\fg_{J} \subseteq \fg$ similarly.
  Note that we have
  \[ D(-\varpi_r - \varpi_s + \gL_0) = F_{w_\circ w_{[1,s]}}D(\varpi_{s-r} + \gL_0).
  \]
  Let $w_\circ^J$ be the longest element of the Weyl group of $\fg_J$.
  Then $F_{w_\circ^J w_{[1,s]}}D(\varpi_{s-r} + \gL_0)$ is a simple $\hfg_{\hat{J}}$\,-module 
  with highest weight $\varpi_{s-r} + \gL_0$,
  and therefore we have
  \[ F_{w_\circ^J w_{[1,s]}}D(\varpi_{s-r} + \gL_0) \cong \bigoplus_{k =0}^{r} V_{J}(\varpi_{r-k} + \varpi_{s-k})
  \]
  as $\fg_J$-modules, where $V_{J}(\nu)$ denotes the simple highest weight $\fg_{J}$-module with highest weight $\nu$. 
  Since $V_{J}(\nu)$ are Demazure modules for $\fg$, $F_{w_\circ w_\circ^J}V_{J}(\nu) = V(\nu)$ holds.
  Hence the assertion is proved.
\end{proof}

We also need the following lemma.

\begin{Lem}\label{Lem:exact}
  There exists an exact sequence
  \begin{align*}
    0 \to L_q(\bm{\varpi}_{r,a}\bm{\varpi}_{s,aq^{s-r+2}}) &\to L_q(\bm{\varpi}_{r,a}) \otimes L_q(\bm{\varpi}_{s,aq^{s-r+2}}) \\
      & \to L_q(\bm{\varpi}_{r-1,aq}\bm{\varpi}_{s+1,aq^{s-r+1}}).
  \end{align*}
\end{Lem}

Assuming this lemma for a moment, we shall complete the proof of Lemma \ref{Lem:fundamentals}.
Take $\bm{\rho}$ as in Lemma \ref{Lem:fundamentals}.
By the results in Subsections \ref{Proof1} and \ref{Proof2}, we have
\[ D(-\varpi_r - \varpi_s + \gL_0) \twoheadrightarrow M(\varpi_r + \varpi_s) 
   \twoheadrightarrow L(\bm{\rho}).
\]
Hence it suffices to show that $\dim L(\bm{\rho}) \ge \dim D(-\varpi_r - \varpi_s + \gL_0)$.
Recall that every fundamental module $L_q(\bm{\varpi}_{i,a})$ is simple as a $U_q(\fg)$-module (in type $C$),
and it follows for $1 \le i <j \le n$ that
\begin{align}\label{eq:tensor}
   V_q(\varpi_i&) \otimes V_q(\varpi_j) \\ &\cong 
    \begin{cases} \bigoplus_{k=0}^i V_q(\varpi_{i-k} + \varpi_{j-k})
                   \oplus \bigoplus_{k= 1}^i V_q(\varpi_{i-k} + \varpi_{n-|n-j-k|}) & \text{if} \ j \le n-1, \\
                  \bigoplus_{k=0}^i V_q(\varpi_{i-k} + \varpi_{n-k}) & \text{if} \ j = n.
    \end{cases}\nonumber
\end{align}
By Theorem \ref{Thm:Classification}, 
$\bm{\rho} = \bm{\varpi}_{r,a} \bm{\varpi}_{s,aq^{\gee(s-r+2)}}$ for some $a \in \C(q)^\times$ 
and $\gee \in \{\pm 1\}$.
Let us assume $\gee = +1$ first. Using
\[ L_q(\bm{\varpi}_{r-1,aq}\bm{\varpi}_{s+1,aq^{s-r+1}}) \hookrightarrow L_q(\bm{\varpi}_{r-1,aq}) 
   \otimes L_q(\bm{\varpi}_{s+1,aq^{s-r+1}}),
\]
we see from Lemma \ref{Lem:exact} and (\ref{eq:tensor}) that
$L_q(\bm{\rho})$ contains $V_q(\varpi_{r-k} + \varpi_{s-k})$ ($0 \le k \le r$) as simple $U_q(\fg)$-module components.
Hence $\dim L(\bm{\rho}) = \dim L_q(\bm{\rho}) \ge \dim D(-\varpi_r - \varpi_s + \gL_0)$ holds by Lemma \ref{Lem:decomp},
as desired.
Since $\dim L_q(\bm{\rho}) = \dim L_q({}^*\!\bm{\rho})$ holds by Lemma \ref{Lem:dual}, the case $\gee = -1$ is also proved.

Now let us prove Lemma \ref{Lem:exact}. % Set $L_1 = L_q(\bm{\varpi}_{r,a})$ and $L_2 = L_q(\bm{\varpi}_{s,aq^{s-r+2}})$.
%Note that the weight set of every finite-dimensional $\ell$-weight module contains at least one dominant $\ell$-weight.
It suffices to show that
\[ \wt_\ell\Big(L_q(\bm{\varpi}_{r,a}) \otimes L_q(\bm{\varpi}_{s,aq^{s-r+2}})\Big) \cap P^+_q 
   = \big\{\bm{\varpi}_{r,a} \bm{\varpi}_{s,aq^{s-r+2}},\bm{\varpi}_{r-1,aq} \bm{\varpi}_{s+1,aq^{s-r+1}}\big\}
\]
and each dominant $\ell$-weight space is $1$-dimensional. 
Assume that $\bm{\rho}_1 \in \wt_{\ell}L_q(\bm{\varpi}_{r,a})$ and $\bm{\rho}_2 \in \wt_{\ell} L_q(\bm{\varpi}_{s,aq^{s-r+2}})$
satisfy $\bm{\rho}_1\bm{\rho}_2 \in P_q^+$.
If $\bm{\rho}_2 \neq \bm{\varpi}_{s,aq^{s-r+2}}$, it follows that
\[ \bm{\rho}_1\bm{\rho}_2 \le \bm{\varpi}_{r,a} \bm{\varpi}_{s,aq^{s-r+2}}
   \bm{\ga}_{s,aq^{s-r+3}}^{-1}
\]
by Lemma \ref{Lem:order} (ii),
%Since $\bm{\varpi}_{i,a} \bm{\varpi}_{j,aq^{j-i+2}}\bm{\ga}_{j,aq^{j-i+3}}^{-1}$ is right-negative, 
and therefore $\bm{\rho}_1\bm{\rho}_2$ is right-negative by Lemma \ref{Lem:order} (i).
Hence we have $\bm{\rho}_2 = \bm{\varpi}_{s,aq^{s-r+2}}$.

We need to show one more lemma.
For $\bm{\nu} \in \wt_\ell L_q(\bm{\varpi}_{r,a})$, 
define $u_i(\bm{\nu}) \in \Z_{\ge 0}$ for $i \in I$ by $\varpi_r - \wt(\bm{\nu}) = \sum_{i \in I} u_i(\bm{\nu}) \ga_i$.
Let $u(\bm{\nu}) = \sum_{i \in I} u_i(\bm{\nu}) \in \Z_{\ge 0}$.

\begin{Lem}
  Let $r \le k \le n$,
  and assume that $\bm{\nu} \in \wt_\ell L_q(\bm{\varpi}_{r,a})$ satisfies $u_k(\bm{\nu}) > 0$.
  Then $\bm{\nu} \le \bm{\varpi}_{r,a} \bm{\ga}_{k,aq^{p(k)}}^{-1}$ holds, where we set $p(k) = k- r + 1 + \gd_{kn}$.
\end{Lem}

\begin{proof}
  We prove the assertion by the induction on $k$.
  The case $k = r$ follows from Lemma \ref{Lem:order} (ii).
  Let $k > r$, and assume that there is an element $\bm{\nu}$ such that $u_k(\bm{\nu})>0$
  and $\bm{\nu} \nleq \bm{\varpi}_{r,a}\bm{\ga}_{k,aq^{p(k)}}^{-1}$.
  We may assume further that $u(\bm{\nu})$ is minimal among such elements.
  By Proposition \ref{Prop:3cond},
  there exist $j \in I$ and $\bm{\nu}' \in \wt_{\ell} L_q(\bm{\varpi}_{r,a})$ satisfying the conditions (i)--(iii) therein.
  If $u_k(\bm{\nu}') > 0$, then $\bm{\nu}'$ also satisfies the assumption of $\bm{\nu}$, 
  which contradicts the minimality of $u(\bm{\nu})$. 
  Hence $j = k$ and $u_k(\bm{\nu}') = 0$ follow.
  On the other hand, since $u_{k-1}(\bm{\nu}) > 0$ obviously holds, we have $u_{k-1}(\bm{\nu}') > 0$.
  Then we see from the weight set of $L_q(\bm{\varpi}_{r,a}) \cong V_q(\varpi_r)$ that
  $u_{k-1}(\bm{\nu}') = 1$ and $u_l(\bm{\nu}') = 0$ for $l \ge k$, and therefore
%  Since $\bm{\nu}' \le \bm{\varpi}_{r,a} \bm{\ga}_{k-1,aq^{p(k-1)}}^{-1}$ holds by the induction hypothesis,
  we have
  \[ \bm{\nu}' \in \bm{\varpi}_{r,a} \bm{\ga}_{k-1,aq^{p(k-1)}}^{-1}
     \prod_{l < k-1, b \in \C(q)^\times} \bm{\ga}_{l,b}^{\Z_{\le 0}}
  \]
  by the induction hypothesis.
  Thus it follows that $\bm{\nu}'_k(u) = 1 - aq^{p(k-1)}u$. Then by Propositions \ref{Prop:sl2} and
  \ref{Prop:reduction}, condition (iii) implies   
  $\bm{\nu} = \bm{\nu}' \bm{\ga}_{k,aq^{p(k-1)}q_k}^{-1}\le \bm{\varpi}_{r,a} \bm{\ga}_{k,aq^{p(k)}}^{-1}$,
  which is a contradiction. The lemma is proved.  
\end{proof}

This lemma implies that, if $\bm{\nu} \in \wt_\ell L_q(\bm{\varpi}_{r,a})$ satisfies
$u_k(\bm{\nu}) > 0$ for some $k > s$, then $\bm{\nu} \bm{\varpi}_{s,aq^{s-r+2}}$ is right-minimal.
Hence $u_k(\bm{\rho}_1) = 0$ holds for all $k > s$.
Moreover $\bm{\rho}_1 \bm{\varpi}_{s,aq^{s-r+2}} \in P_q^+$ implies $\wt(\bm{\rho}_1) + \varpi_s \in P^+$.
Hence we have 
\[ \wt(\bm{\rho}_1) \in \{ \varpi_r, \varpi_{r-1} -\varpi_{s} + \varpi_{s+1}\}.
\]
Then we see from \cite[Theorem 2.7]{MR2242948} that 
\[ \bm{\rho}_1 \in \{ \bm{\varpi}_{r,a}, \bm{\varpi}_{r-1,aq}\bm{\varpi}_{s,aq^{s-r+2}}^{-1}\bm{\varpi}_{s+1,aq^{s-r+1}}\}
\]
and $\dim L_q(\bm{\varpi}_{r,a})_{\bm{\rho_1}} = 1$.
Hence we have
\[ \bm{\rho}_1 \bm{\rho}_2 = \bm{\rho}_1 \bm{\varpi}_{s,aq^{s-r+2}} \in
   \{ \bm{\varpi}_{r,a}\bm{\varpi}_{s,aq^{s-r+2}}, \bm{\varpi}_{r-1,aq}\bm{\varpi}_{s+1,aq^{s-r+1}}\},
\] 
and $\dim\Big(L_q(\bm{\varpi}_{r,a}) \otimes L_q(\bm{\varpi}_{s,aq^{s-r+2}})\Big)_{\bm{\rho}_1\bm{\rho}_2} = 1$, as desired.
The proof is complete.\\

\noindent {\bf Acknowledgment.} 
A part of this work was carried out during the visit of the author to University of California, Riverside.
He would like to express his gratitude to V.\ Chari for her hospitality and fruitful discussion.
He also thank R.\ Kodera for helpful comments.
This work was supported by World Premier International Research Center Initiative (WPI Initiative), MEXT, Japan.

%\bibliographystyle{alpha}
%\bibliography{bibliography-minimal_affinization}

\def\cprime{$'$} \def\cprime{$'$}

\end{document}